\documentclass[12pt,reqno]{amsart}

\usepackage{comment}

\usepackage{url}
\usepackage{amsfonts,amsmath,amsthm}
\usepackage{amssymb,epsfig}
\usepackage{enumerate} 
\usepackage[sc,osf]{mathpazo}   % With old-style figures and real smallcaps.
\usepackage[euler-digits,small]{eulervm}
\AtBeginDocument{\renewcommand{\hbar}{\hslash}}
\usepackage[text={425pt,650pt},centering]{geometry}
\usepackage{graphicx}
\usepackage{epsfig}
\usepackage{tikz}
\usepackage{caption}
\usepackage{color} %color
\definecolor{vert}{rgb}{0,0.6,0}

\numberwithin{figure}{section}

\theoremstyle{plain}
\newtheorem{thm}{Theorem}[section]

\newtheorem{lem}[thm]{Lemma}
\newtheorem{cor}[thm]{Corollary}
\newtheorem{prop}[thm]{Proposition}
\theoremstyle{remark}
\newtheorem{rem}{\bf{Remark}}
\numberwithin{equation}{section}

\newcommand{\R}{\mathbb{R}}

\begin{document}

\title[Rate of convergence]
{Rate of convergence for periodic homogenization of convex Hamilton--Jacobi equations in one dimension}

\author[Son N.T. Tu]
{Son N.T. Tu}

\thanks{}

\begin{abstract}
Let $u^\varepsilon$ and $u$ be viscosity solutions of the oscillatory Hamilton-Jacobi equation and its corresponding effective equation. Given bounded, Lipschitz initial data, we present a simple proof to obtain the optimal rate of convergence $\mathcal{O}(\varepsilon)$ of $u^\varepsilon \rightarrow u$ as $\varepsilon \rightarrow 0^+$ for a large class of convex Hamiltonians $H(x,y,p)$ in one dimension. This class includes the Hamiltonians from classical mechanics with separable potential. The proof makes use of optimal control theory and a quantitative version of the ergodic theorem for periodic functions in dimension $n = 1$.
\end{abstract}

\address[S. N.T. Tu]
{
Department of Mathematics, 
University of Wisconsin Madison, 480 Lincoln  Drive, Madison, WI 53706, USA}
\email{thaison@math.wisc.edu}

%\date{\today}

\maketitle

%\tableofcontents

\section{Introduction}
We first give a brief description of the periodic homogenization theory for Hamilton--Jacobi equations in the framework of viscosity solutions (see \cite{Bardi1997,Le2017,Crandall1992,Lions1982}).
The oscillatory Hamilton--Jacobi equation is given by the following Cauchy problem with parameter $\varepsilon$:
\begin{equation}\label{C_varepsilon}
\begin{cases}
u^\varepsilon_t + H\left(x,\frac{x}{\varepsilon},Du^\varepsilon\right) &= 0 \quad\quad\;\,\text{in}\quad \mathbb{R}\times [0,\infty)\\
\qquad\qquad\quad u^\varepsilon(x,0) &= u_0(x) \;\;\text{on}\quad \mathbb{R},
\end{cases} \tag{C$_\varepsilon$}
\end{equation}
where the initial data $u_0$ is contained in $\mathrm{BUC}(\mathbb{R}^n)$, the set of bounded uniformly continuous functions on $\mathbb{R}^n$. Given a Hamiltonian $H(x,y,p) \in \mathrm{C}(\mathbb{R}^n \times \mathbb{R}^n \times \mathbb{R}^n)$ satisfying some conditions (H1)--(H4) below, define the effective Hamiltonian as follows: For each $(x,p) \in \R^n \times \R^n$, let $\overline{H}(x,p) \in \mathbb{R}$ be the unique constant for which the cell (ergodic) problem 
\begin{equation}\label{cell}
H\big(x,y,p+D_yv(y)\big) = \overline{H}(x,p)\qquad\text{in}\;\mathbb{T}^n \tag{CP}
\end{equation}
has a continuous viscosity solution $v(y) = v(y; x,p)$. 
That such a constant exists and is unique is proven in \cite{Lions1986} and \cite{Evans1989,Evans1992}. It is worth mentioning that in general the solution $v(y;x,p)$ to the cell problem \eqref{cell} is not unique even up to the addition of a constant. The effective Hamilton-Jacobi equation corresponding to \eqref{C_varepsilon} is given by the following Cauchy problem:
\begin{equation}\label{C}
\begin{cases}
u_t + \overline{H}\left(x,Du\right) &= 0 \quad\quad\;\,\text{in}\quad \mathbb{R}\times [0,\infty)\\
\quad\;\;\,\qquad u(x,0) &= u_0(x) \;\;\text{on}\quad \mathbb{R}.
\end{cases}\tag{C}
\end{equation}
Some papers treating the properties of the effective Hamiltonian $\overline{H}$ are \cite{Camilli2008a,Concordel1996,Concordel1997,Luo2016,Qian2017}, and the references given therein.

The theory of periodic homogenization studies the behavior of viscosity solutions $u^\varepsilon \in C(\mathbb{R}^n \times [0,\infty)$ to \eqref{C_varepsilon} as the period of oscillation $\varepsilon$ approaches $0^+$. 
The first results in the theory of periodic homogenization were proved under the following assumptions on the Hamiltonian $H = H(x,y,p) \in \mathrm{C}(\mathbb{R}^n\times\mathbb{T}^n\times \mathbb{R}^n)$:
\begin{itemize}
\item[(H1)] For each $(x,p)\in \mathbb{R}^n\times \mathbb{R}^n$, $y\mapsto H(x,y,p)$ is $\mathbb{Z}^n$-periodic.
\item[(H2)] $p\mapsto H(x,y,p)$ is uniformly coercive in $(x,y)\in \mathbb{R}^n\times\mathbb{T}^n$. That is, 
\begin{equation*}
\lim_{|p|\rightarrow +\infty} \left( \inf_{(x,y)\in \mathbb{R}^n\times\mathbb{T}^n} H(x,y,p)\right) =+\infty.
\end{equation*}
Here $\mathbb{T}^n=\mathbb{R}^n\backslash \mathbb{Z}^n$.
\item[(H3)] $\sup \left\lbrace |H(x,y,p)|: (x,y)\in \mathbb{R}^{2n}, |p|\leq R \right\rbrace < \infty$ for all $R>0$.
\item[(H4)] For each $R>0$, there exists $\omega_R(\cdot)\in \mathrm{C}([0,\infty))$, with $\omega_R(0) = 0$, such that for all $x,y\in \mathbb{R}^n, p,q \in B(0,R)$ then
\begin{equation*}
|H(x,y,p) - H(x,y,q)| \leq \omega_R(|p-q|)
\end{equation*}
where $B(0,R)$ denotes the open ball centered at 0 with radius $R$ in $\mathbb{R}^n$.
\end{itemize}

Under the assumptions (H1)--(H4), the viscosity solutions $u^\varepsilon$ converge to a limit $u$ locally uniformly on $\mathbb{R}^n \times [0,\infty)$, where $u$ is a viscosity solution to the effective equation \eqref{C}. This was first proved by P.-L. Lions, G. Papanicolau and S.R.S. Varadhan \cite{Lions1986} in the case that $H$ is independent of $x$, namely $H(x,y,p) = H(y,p)$. The more general case in which $H = H(x,y,p)$ can depend on $x$ was established later by L. C. Evans \cite{Evans1989,Evans1992}, who developed the perturbed test functions method for studying the homogenization problem in the framework of viscosity solutions.

The rate of convergence of $u^\varepsilon\rightarrow u$ was first studied by I. Capuzzo-Dolcetta and H. Ishii in \cite{Capuzzo-Dolcetta2001} using a PDE approach. They consider the stationary problem
\begin{equation}\label{S_eps}
w^\varepsilon(x) + H\left(x,\frac{x}{\varepsilon},Dw^\varepsilon(x)\right) = 0 \qquad\text{in}\;\mathbb{R}^n. \tag{S$_\varepsilon$}
\end{equation}
As $\varepsilon\rightarrow 0$, $w^\varepsilon\rightarrow w$ locally uniformly on $\mathbb{R}^n$ and $w$ solves the effective equaition
\begin{equation}\label{S}
w(x)+ \overline{H}\left(x,Dw(x)\right) = 0 \qquad\text{in}\;\mathbb{R}^n. \tag{S}
\end{equation}
Under this stationary setting, the authors of \cite{Capuzzo-Dolcetta2001} establish the rate of convergence is \textit{at least} $\mathcal{O}(\varepsilon^{1/3})$ for general (including nonconvex) Lipschitz Hamiltonians under quite general assumptions. In the case that $H(x,y,p) = H(y,p)$, Capuzzo-Dolcetta and Ishii obtain the rate of convergence $\mathcal{O}(\varepsilon)$ of $w^\varepsilon$ to $w$ by a simple comparison argument. Their approach can be easily adjusted to handle the Cauchy problem \eqref{C_varepsilon} giving the same rate $\mathcal{O}(\varepsilon^{1/3})$. This approach is quite robust, and it works for various different situations. Another example occurs in \cite{Marchi2009}, where C. Marchi considers the case where $H$ depends on more scales, and establishes the rate $\mathcal{O}(\varepsilon^{1/3})+\omega(\varepsilon)$ for some modulus of continuity of $H$ using the method of \cite{Capuzzo-Dolcetta2001},

Heuristically, the rate of convergence $\mathcal{O}(\varepsilon)$ seems to be optimal. By using an ansatz $u^\varepsilon = u^0 + \varepsilon u^1 + \varepsilon^2 u^2+\ldots$ and plugging it into \eqref{C_varepsilon}, we can derive the following two--scale asymptotic expansion (see \cite{Le2017, Lions1986, Mitake2018}),
\begin{equation}\label{eqn:heuristic}
u^\varepsilon(x,t) \approx u(x,t) + \varepsilon v\left(\frac{x}{\varepsilon};x, Du(x,t)\right) + \mathcal{O}(\varepsilon^2),
\end{equation}
in which the rate of convergence looks like $\mathcal{O}(\varepsilon)$. However, it is hard to justify \eqref{eqn:heuristic} rigorously as the solution $u(x,t)$ to \eqref{C} is only Lipschitz in $(x,t)$, and is usually not $\mathrm{C}^1$. Also, the solution $v$ to the ergodic problem \eqref{cell} is not unique even up to the addition of a constant (Example 6.1 in \cite{Le2017} or Proposition 5.4 in \cite{Lions1982}).

%The authors of \cite{Mitake2018} give an example showing that in general the optimal rate cannot be better than $\mathcal{O}(\varepsilon)$. We recall it here:

%\begin{prop}[Proposition 4.3, \cite{Mitake2018}]\label{propop} Let $n=1$ and $H(y,p) = \frac{1}{2}|p|^2 +  V(y)$ where $V\in \mathrm{C}(\mathbb{T})$ with $\max_{\mathbb{T}}V = 0$ and $V\leq -1$ in $\left[-\frac{1}{3},\frac{1}{3}\right]$. Then in this case $u\equiv 0$,  $\Vert u^\varepsilon\Vert_{L^\infty(\mathbb{R}\times [0,\infty))} \leq C\varepsilon$ and $u^\varepsilon(0,1) \geq \frac{1}{6}\varepsilon$ for all $\varepsilon\in (0,1)$.
%\end{prop}

%It is worth mentioning that in general $v(y;x,p)$ is not unique even up to adding a constant. More research in understanding the effective Hamiltonians $\overline{H}$ is reported in \cite{Camilli2008a,Concordel1996,Concordel1997,Luo2016,Qian2017} and the references therein.

%The main goal of this paper is to obtain the optimal rate of convergence of $u^\varepsilon$ to $u$ in one dimension. More precisely, we find an optimal bound for $\Vert u^\varepsilon -u\Vert_{L^\infty([-R,R]\times [0,T])}$ for any given $R,T>0$ as $\varepsilon\rightarrow 0^+$.

Recently, H. Mitake, H. V. Tran and Y. Yu established in \cite{Mitake2018}  that the rate $\mathcal{O}(\varepsilon)$ is optimal in the case that the dimension $n=1$ and the Hamiltonian $H$ is convex and independent of $x$. They provide the following example of a family of $u^\varepsilon$'s that converge to $u$ at the strict rate of $O(\varepsilon)$:

\begin{prop}[Proposition 4.3, \cite{Mitake2018}]\label{propop} Let $n=1$ and $H(y,p) = \frac{1}{2}|p|^2 +  V(y)$ where $V\in \mathrm{C}(\mathbb{T})$ with $\max_{\mathbb{T}}V = 0$ and $V\leq -1$ in $\left[-\frac{1}{3},\frac{1}{3}\right]$. Then in this case $u\equiv 0$,  $\Vert u^\varepsilon\Vert_{L^\infty(\mathbb{R}\times [0,\infty))} \leq C\varepsilon$ and $u^\varepsilon(0,1) \geq \frac{1}{6}\varepsilon$ for all $\varepsilon\in (0,1)$.
\end{prop}
Proposition 1.1 and other important results in higher dimensional spaces are proved in \cite{Mitake2018} using tools from dynamical systems and weak KAM theory. 

Mitake, Tran, and Yu also present in \cite{Mitake2018} an essential obstacle to improving the convergence rate $\mathcal{O}(\varepsilon^{1/3})$ by the method used by Capuzzo-Dolcetta and Ishii in \cite{Capuzzo-Dolcetta2001}. %Precisely, the authors of \cite{Capuzzo-Dolcetta2001} approximated the expansion \eqref{eqn:heuristic} by using the ergodic problems and their discounted approximations. 
More precisely, for each $(x,p)\in \mathbb{R}^n\times \mathbb{R}^n$, instead of using $v(y;x,p)$ directly in \eqref{eqn:heuristic}, the authors of \cite{Capuzzo-Dolcetta2001} use $v^\lambda(y) = v^\lambda(y;x,p)$ as the unique solution to the following discount problem
\begin{equation}\label{discount}
\lambda v^\lambda(y) + H\left(x,y,p+D_yv^\lambda (y)\right) = 0 \qquad\text{in} \;\mathbb{T}^n, \tag{S$_\lambda$}
\end{equation}
and approximate $Du(x,t)$ by $\frac{x-y}{\varepsilon^\beta}$ in \eqref{eqn:heuristic} using the doubling variable method. By optimizing $\lambda$ and $\beta$, $\mathcal{O}(\varepsilon^{1/3})$ is the best convergence rate that can be obtained. In order to improve the convergence rate, it is
necessary to have a nice selection of viscosity solutions $v(\cdot;x,p)$ to the ergodic problem \eqref{cell} with respect to $(x,p)$, so that one can use directly $v(y;x,p)$ instead of $v^\lambda(y;x,p)$ in \eqref{eqn:heuristic}. In the case that $H(x,y,p) = H(y,p)$, assume that
\begin{equation}\label{impr}
\begin{cases}
\text{For each}\;p\in \mathbb{R}^n \;\text{there exists a solution}\;v(\cdot;p)\;\text{of}\;\eqref{cell} \\
\text{such that}\; p\mapsto v(\cdot; p)\;\text{is Lipschitz}.
\end{cases}
\end{equation}
Then, the convergence rate can be improved from $\mathcal{O}(\varepsilon^{1/3})$ to $\mathcal{O}(\varepsilon^{1/2})$, as one needs only introduce one parameter into the doubling variable formulation (see Section 7.2 in \cite{Hung2019}) instead of two parameters as before. However, condition \eqref{impr} is quite restrictive in general and does not always hold (see Section 5 in \cite{Mitake2018}). 

Closely related to the results outlined above for the problem \eqref{C_varepsilon} are the recent developments in the case of the viscous Hamilton{--}Jacobi equations. Let $H = H(x,y,p,X): \mathbb{R}^n\times\mathbb{R}^n\times\mathbb{R}^n\times \mathbb{S}^n \rightarrow \mathbb{R}$ be the Hamiltonian that is $\mathbb{Z}^n$-periodic in the $y$ variable, and $\mathbb{S}^n$ denotes the set of $n\times n$ symmetric matrices. The associated viscous Cauchy problem is
\begin{equation}\label{C_varepsilonstar}
\begin{cases}
u^\varepsilon_t + H\left(x,\frac{x}{\varepsilon},Du^\varepsilon, D^2u^\varepsilon \right) &= 0 \quad\quad\;\,\text{in}\quad \mathbb{R}\times [0,\infty)\\
\qquad\qquad\quad u^\varepsilon(x,0) &= u_0(x) \;\;\text{on}\quad \mathbb{R},
\end{cases} \tag{C$_\varepsilon^*$}
\end{equation}
One can find the effective Hamiltonian $\overline{H}$ with a method similar to that used in the non-viscous case and obtain a solution $u$ to the Cauchy problem associated to $\overline{H}$, such that the solutions $u^\varepsilon$ to \eqref{C_varepsilonstar} converge locally uniformly to $u$ (see \cite{Camilli_2009,Evans1992}). The following analogous results on the rate of convergence of $u^\varepsilon\rightarrow u$ for the viscous Hamilton--Jacobi equation below are important to note:
\begin{itemize}
\item In the stationary setting, F. Camilli and C. Marchi (\cite{Camilli_2009}) show that the rate is $\mathcal{O}(\varepsilon)$ if $H = H(y,p,X)$. It can be upgraded to $\mathcal{O}(\varepsilon^2)$ if $H = H(y,X)$.
\item F. Camilli, C. Annalisa and C. Marchi (\cite{Camilli2016}) show that the rate is $\mathcal{O}(\varepsilon)$ for the vanishing viscosity problem $u^\varepsilon + H\left(\frac{x}{\varepsilon},Du^\varepsilon, \varepsilon D^2u^\varepsilon\right) = 0$ in $\mathbb{R}^n$. 
\item For the Cauchy problem $u_t^\varepsilon + H\left(\frac{x}{\varepsilon},Du^\varepsilon, \varepsilon D^2u^\varepsilon\right) = 0$ in $\mathbb{R}^n\times (0,\infty)$ with initial data $u(x,0) = g(x)$ on $\mathbb{R}^n$, S. Kim and K.-A. Lee (\cite{Kim2017a}) obtain high order rates of convergence for special chosen initial data. 
\end{itemize}
In both situations, viscous and nonviscous, the case when $H$ depends on $x$ is significantly harder. In particular, the methods used in \cite{Camilli2016,Camilli_2009} provide the rate $\mathcal{O}(\varepsilon^\alpha)$ for some $\alpha<1$.

%In this paper, we obtain the rate $\mathcal{O}(\varepsilon)$ for the Cauchy problem with a class of Hamiltonians that depends on $x$ in one dimension (Theorem \ref{general case}).

We refer to \cite{Camilli_2009,Camilli2016,Kim2017a} and the references therein for more related results on the viscous case. See also \cite{Armstrong2014,Caffarelli2010,Luo2011,Mitake2014} and the references therein for related results to the rate of convergence of Hamilton--Jacobi equations in stochastic homogenization and other settings.
\medskip

%The rate of convergence of $u^\varepsilon$ to $u$ in different settings also received a lot of attentions. See \cite{Armstrong2014, Caffarelli2010a} and the references therein for related rate of convergence in the stochastic setting. 

%In \cite{Armstrong2014} with the stochastic setting, by using a novel integration of probabilistic and PDE techniques, the authors established the rate $-\mathcal{O}(\varepsilon^{1/8-\delta}) \leq u^\varepsilon - u\leq \mathcal{O}(\varepsilon^{1/5-\delta})$ for any $\delta>0$. 

When $H$ is convex and depends on $x$, the situation is more complicated and requires a harder analysis of the dynamics of optimal paths in the optimal control formula. Up to now, the best-known convergence rate in this setting is $\mathcal{O}(\varepsilon^{1/3})$, obtained in \cite{Capuzzo-Dolcetta2001}. The main goal of this paper is to obtain the optimal rate of convergence of $u^\varepsilon \rightarrow u$ in one dimension: more precisely, to obtain an optimal bound for $\Vert u^\varepsilon -u\Vert_{L^\infty([-R,R]\times [0,T])}$ for any given $R,T>0$ as $\varepsilon\rightarrow 0^+$. Our paper is the first work that improves the rate of convergence $u^\varepsilon \rightarrow u$ for \eqref{C_varepsilon} to $\mathcal{O}(\varepsilon)$ in the one-dimensional case, as far as we know. Our method develops the results of \cite{Mitake2018} further and uses deep dynamical properties of optimal paths in the optimal control formula. Higher-dimensional cases will be investigated in future works. We state our main results precisely:

%However, as pointed out in \cite{Mitake2018} by H. Mitake, H. V. Tran and Y. Yu, it is hard to justify this expansion rigorously due to two reasons:
%\begin{itemize}
%\item In general, there does not even exists a continuous selection of $v(\cdot;x,p)$ with respect to $p$, let alone Lipschitz continuous selection. 
%\item The solution $u(x,t)$ to \eqref{C} is only Lipschitz in $(x,t)$, and is usually not $\mathrm{C}^1$.
%\end{itemize}

\subsection{Main results}
In this paper, we consider the one dimensional case $n=1$ and the convex Hamiltonian is of the form:
\begin{equation*}
H(x,y,p) = H(p) + V(x,y) \qquad\text{for all}\qquad (x,y,p) \in \mathbb{R}\times\mathbb{T}\times\mathbb{R}.
\end{equation*}

We present a simplified proof for obtaining the rate of convergence using optimal control theory. %It relies on two facts: 
%\begin{itemize}
%\item[(i)] the oscillatory solutions $u^\varepsilon$ have an explicit variational formulation, in terms of minimizing an action functional, with respect to absolutely continuous trajectories, and
%\item[(ii)] there exists a quantitative version of the ergodic theorem for periodic functions in dimension $n = 1$ (Lemma \ref{Lemma on average - C1 version}, see also \cite{Neukamm2017}).
%\end{itemize}
The main theorem is as follows:
\begin{thm}\label{general case} Let $n=1$ and $H(x,y,p) = H(p) + V(x,y)+C_0$ where $C_0$ is a constant, $H\in\mathrm{C}^2(\mathbb{R},[0,\infty))$ is strictly convex with $\min_{p\in \mathbb{R}} H(p) = H(0) = 0$. Define $G_1 = \left(H'|_{[0,\infty)}\right)\circ \left(H|_{[0,\infty)}\right)^{-1}$ and $G_2 = \left(H'|_{(-\infty,0]}\right)\circ \left(H|_{(-\infty,0]}\right)^{-1}$. Assume $\mathrm{(H1)}$--$\mathrm{(H4)}$, $V(x,y)$ is continuously differentiable in $x$ variable for each $y\in \mathbb{T}$, and:
\begin{enumerate}
\item[$\mathrm{(A0)}$]  
\begin{equation}
\limsup_{p\rightarrow 0} \left|\frac{H''(p)}{H'(p)}\sqrt{H(p)}\right| < \infty.
\end{equation}
\item[$\mathrm{(A1)}$] $\max_{\mathbb{R}\times \mathbb{T}} V(x,y) = 0$, there exists $y_0 \in \mathbb{T}$ such that $V(x,y_0) = 0$ for all $x\in \mathbb{R}$.
\end{enumerate}
For each compact interval $I\subset\mathbb{R}$  and $i=1,2$ we have:
\begin{enumerate}
\item[$\mathrm{(A2)}$]
\begin{equation*}
\limsup_{r\rightarrow 0^+} \left\lbrace |V_x(x,y)|.\frac{|G_i'(r-V(x,y))|}{|G_i(r-V(x,y))|}\;: \;(x,y)\in I\times\mathbb{T}\right\rbrace<\infty.
\end{equation*}
\item[$\mathrm{(A3)}$] 
\begin{equation*}
\sup_{(x,y)\in I\times\mathbb{T}}\left|\frac{V_x(x,y)}{G_i\left(|V(x,y)|\right)}\right| <\infty.
\end{equation*}
\item[$\mathrm{(A4)}$] 
\begin{equation*}
\limsup_{r\rightarrow 0^+}\left( \frac{\displaystyle\max_{x\in I} \int_0^1 \frac{dy}{|G_i(r-V(x,y))|}}{\displaystyle\min_{x\in I} \int_0^1 \frac{dy}{|G_i(r-V(x,y))|}}\right) <\infty.
\end{equation*}
\end{enumerate}
If $u_0\in \mathrm{Lip}(\mathbb{R})\cap\mathrm{BUC}(\mathbb{R})$ then for any $R,T>0$ we have
\begin{equation}\label{ge.rate}
\left\Vert u^\varepsilon - u\right\Vert_{L^\infty\left([-R,R]\times [0,T]\right)} \leq C\varepsilon
\end{equation}
where $C$ is a constant depends only on $R,T,\mathrm{Lip}(u_0)$, $H(p)$ and $V(x,y)$.
\end{thm}

To help the readers better understand the main ideas of the paper, we
consider the classical mechanics Hamiltonian in Theorem \ref{simple case}. Theorem \ref{general case} is an adaptation of these ideas to the general situation with new technical changes to overcome the arising difficulties.

\begin{thm}\label{simple case} Assume $n=1$ and $H(x,y,p) = \frac{1}{2}|p|^2 + V(x,y)$ where $V$ is of the separable form $V(x,y) = a(x)b(y)+C_0$ where $C_0$ is a constant and
\begin{itemize}
\item[(i)] $a(x)\in \mathrm{C}^1(\mathbb{R})$ is bounded with $a(x) > 0$ for all $x\in \mathbb{R}$,
\item[(ii)] $b(y)\in \mathrm{C}(\mathbb{T})$ and $\max_{y\in \mathbb{T}} b(y) = 0$.
\end{itemize}
Assume $u_0\in \mathrm{Lip}(\mathbb{R})\cap\mathrm{BUC}(\mathbb{R})$, then for each $R,T>0$ we have
\begin{equation}\label{rate}
\left\Vert u^\varepsilon - u\right\Vert_{L^\infty\left([-R,R]\times [0,T]  \right)} \leq C\varepsilon
\end{equation}
where $C$ is a constant depends on $R,T,\mathrm{Lip}(u_0)$, $a(x)$ and $\max |b(y)|$.
\end{thm}

\begin{rem} In Theorem \ref{simple case} if $V(x,y) = V(y)$ does not depend on $x$, then we can choose $C$ explicitly as $C = 2\Vert u_0'\Vert_{L^\infty(\mathbb{R})}+ 8\Vert V\Vert_{L^\infty}^{1/2}$. As a consequence, the convergence is uniform in the sense that $\Vert u^\varepsilon  - u\Vert_{L^\infty(\mathbb{R}\times [0,\infty))} \leq C\varepsilon$ (Section 2 and Remark \ref{remark4}).
\end{rem}

\begin{rem} Let us give some quick comments on the assumptions of Theorem \ref{general case}. 
\begin{itemize}
\item[(i)] The assumptions (A2)--(A4) are technical assumptions that are needed for the arguments to work. These assumptions are natural in the sense that they are satisfied by a large class of interesting Hamiltonians (cf. Corollary \ref{corollary simplified}).
\item[(ii)] Assumption (A1) plays a key role in establishing the result. Roughly speaking, the rate of convergence of $u^\varepsilon$ to $u$ is related to the asymptotic behavior of its corresponding minimizer path via an optimal control formulation as in  \eqref{ge.u^eps.def}. Any minimizer path conserves the total energy as in \eqref{ge.H=r}. Assumption (A1) implies that any minimizer with negative total energy is uniformly bounded independent of $\varepsilon>0$. 
\item[(iii)] Condition (A0) is satisfied for a vast class of strictly convex $\mathrm{C}^2$ Hamiltonians, including those with $H''(0) > 0$, $H\in \mathrm{C}^3$, or $|p|^\gamma$ with $\gamma \geq 2$ (Lemma \ref{key lemma}).
%\item[(iv)] If $H''(0) > 0$ and $\sup_{I\times \mathbb{T}}\left|\frac{V_x(x,y)}{V(x,y)}\right|$ is bound for any compact interval $I\subset\mathbb{R}$, then (A2), (A3) hold (Corollary \ref{corollary simplified}). 
\end{itemize}
\end{rem}

\begin{rem} If $V(x,y) = V(y)$ does not depend on $x$, then assumptions (A1)--(A4) automatically hold, while (A0) is satisfied after approximating $H$ with uniformly convex Hamiltonians. Indeed, the method can be used to get the result for general convex Hamiltonians. We thus recover Theorem 1.3 in \cite{Mitake2018} and the convergence is uniform in this case. By Proposition \ref{propop}, the rate $\mathcal{O}(\varepsilon)$ is optimal.
\end{rem}
The following corollary gives some nice examples in which (A1)--(A4) hold, and Theorem \ref{general case} applies.
\begin{cor}\label{corollary simplified} If $H(x,y,p) = H(p) + V(x,y)$ where $H(p)\geq H(0) = 0$ such that:
\begin{itemize}
\item $H(p)\in \mathrm{C}^2(\mathbb{R})$ is strictly convex with $H''(0)>0$, or $H(p) = |p|^\gamma$ where $\gamma\geq 2$.
\item $\max_{\mathbb{R}\times \mathbb{T}} V(x,y) = 0$, there exists $y_0 \in \mathbb{T}$ such that $V(x,y_0) = 0$ for all $x\in \mathbb{R}$.
\item For every compact interval $I\subset\mathbb{R}$ then $\alpha_I f_I(y) \leq |V(x,y)|\leq \beta_I f_I(y)$ for $\alpha_I,\beta_I>0, f_I\in \mathrm{C}(\mathbb{R},[0,\infty))$ and
\begin{equation}\label{rmk.3.2}
\sup_{(x,y)\in I\times\mathbb{T}}\left|\frac{V_x(x,y)}{V(x,y)}\right| \leq C_I< \infty.
\end{equation}
\end{itemize}
If $u_0\in \mathrm{Lip}(\mathbb{R})\cap\mathrm{BUC}(\mathbb{R})$ then for any $R,T>0$ we have
\begin{equation*}
\left\Vert u^\varepsilon - u\right\Vert_{L^\infty\left([-R,R]\times [0,T]\right)} \leq C\varepsilon
\end{equation*}
where $C$ is a constant depends only on $R,T,\mathrm{Lip}(u_0)$, $H(p)$ and $V(x,y)$.
\end{cor}

\subsection{Organization of the paper} The paper is organized as follows. Section 2 is devoted to the proof of Theorem \ref{simple case}. In section 3 we only sketch the proof of Theorem \ref{general case} with highlights on major technical difficulties, and where assumptions (A0)--(A4) appear since the ideas are similar to the classical mechanics' case. We provide proofs of some lemmas in the Appendix for the reader's convenience.

%\begin{rem} With $H(p) = |p|^\gamma$ where $\gamma > 1$, an example of $V(x,y)$ that satiesfies all assumptions of Theorem \ref{general case} but not in the separable form is 
%\begin{equation*}
%V(x,y) = \sum_{i=1}^m a_i(x,y)b_i(y)
%\end{equation*}
%where $a_i\in \mathcal{A}$, $b_i\in \mathcal{B}$ that satisfies $b_i(y_0) = 0$ for all $i=1,2,\ldots m$ for some $m\in \mathbb{N}$, where
%\begin{align*}
%\mathcal{A}:&= \left\lbrace a(x,y): \mathbb{R}\times \mathbb{T}\rightarrow (0,+\infty)\; \text{is continuous and Lipschitz in}\;x \right\rbrace,\\
%\mathcal{B}:&= \left\lbrace b(y): \mathbb{T}\rightarrow (-\infty,0]\;\;\text{is continuous} \right\rbrace.
%\end{align*}
%\end{rem}

\section{Classical Mechanics Hamiltonian setting}

\begin{proof}[Proof of Theorem \ref{simple case}] We observe that the estimate \eqref{rate} does not depend on the smoothness of $b(\cdot)$, by approximation, without loss of generality we can assume that $V\in \mathrm{C}^2(\mathbb{R}\times\mathbb{T})$. Also, by replacing $u$ by $u+C$ we can normalize that $C_0 =0$. Let us fix $R,T>0$, $\varepsilon\in (0,1)$ and $(x_0,t_0)\in [-R,R]\times [0,T]$, thanks to the optimal control formula (see \cite{Bardi1997,Le2017}) we have
\begin{align}\label{u^epsilon-def}
u^\varepsilon(x_0,t_0)= \inf_{\eta\,\in \mathcal{T}} \left\lbrace \varepsilon\int_0^{\varepsilon^{-1}t_0} \left(\frac{|\dot{\eta}(s)|^2}{2} - V\left(\varepsilon \eta(s),\eta(s)\right)\right)ds + u_0\left(\varepsilon\eta(\varepsilon^{-1}t_0)\right)\right\rbrace,
\end{align}
where $\mathcal{T} =\big\lbrace \eta(\cdot)\in \mathrm{AC}\left(\left[0,\varepsilon^{-1}t_0\right]\right), \varepsilon\eta(0) = x_0\big\rbrace$. Here $\mathrm{AC}([a,b])$ denotes the set of absolutely continuous functions from $[a,b]$ to $\mathbb{R}$. Let $\eta_\varepsilon(\cdot)\in \mathcal{T}$ be a minimizer to the optimization problem \eqref{u^epsilon-def}, it is clear that $\eta_\varepsilon(\cdot)$ must satisfy the following Euler-Lagrange equation
\begin{equation}\label{E-L}
\begin{cases}
\ddot{\eta}_\varepsilon(s) &= -\nabla  V\big(\varepsilon\eta_\varepsilon(s),\eta_\varepsilon(s)\big)\cdot(\varepsilon,1)\qquad\text{on}\qquad \left(0,\varepsilon^{-1}t_0\right), \\
\eta_\varepsilon(0) &= \varepsilon^{-1}x_0.
\end{cases}
\end{equation}
Here $
\nabla V$ means the full gradient of $V$. In particular, this implies the following conservation of energy:
\begin{equation*}
\frac{d}{ds}\left(\frac{|\dot{\eta}_\varepsilon(s)|^2}{2} + V\left(\varepsilon\eta_\varepsilon(s),\eta_\varepsilon(s)\right)\right) = \dot{\eta}_\varepsilon(s)\Big(\ddot{\eta}_\varepsilon(s) +\nabla V\left(\varepsilon \eta_\varepsilon(s),\eta_\varepsilon(s)\right)\cdot (\varepsilon, 1)\Big) = 0
\end{equation*}
for all $s\in \left(0,\varepsilon^{-1}t_0\right)$. There exists a constant $r = r(\eta_\varepsilon) \in \left[V(0,0),+\infty\right)$ such that
\begin{equation}\label{H=r}
\frac{|\dot{\eta}_\varepsilon(s)|^2}{2} +  V\left(\varepsilon \eta_\varepsilon(s), \eta_\varepsilon(s)\right) = r \qquad\text{for all}\qquad s\in(0,\varepsilon^{-1}t_0).
\end{equation}
For each $r\in \left[V(x_0,\varepsilon^{-1}x_0),\infty\right)$ the Euler-Lagrange equation \eqref{E-L} is
\begin{equation}\label{full-E-L}
\begin{cases}
\;\ddot{\eta}_\varepsilon(s)\; = -\nabla V\big(\varepsilon\eta_\varepsilon(s),\eta_\varepsilon(s)\big)\cdot(\varepsilon,1)\qquad\text{on}\qquad \left(0,\varepsilon^{-1}t_0\right), \\
|\dot{\eta}_\varepsilon(0)| = \sqrt{2(r-V(x_0,\varepsilon^{-1}x_0))},\\
\;\eta_\varepsilon(0)\; =  \varepsilon^{-1}x_0.
\end{cases}
\end{equation}
For simplicity, let us define the action functional
\begin{equation*}
A^\varepsilon[\eta] = \varepsilon\int_0^{\varepsilon^{-1}t_0} \left(\frac{|\dot{\eta}(s)|^2}{2} - V\left(\varepsilon \eta(s),\eta(s)\right)\right)ds + u_0\left(\varepsilon\eta\left(\varepsilon^{-1}t_0\right)\right)
\end{equation*}
for $\eta(\cdot)\in \mathcal{T}$. Thanks to the conservation of energy \eqref{H=r}, the optimization problem \eqref{u^epsilon-def} is equivalent to
\begin{equation}\label{u^e.def.e:1}
u^\varepsilon(x_0,t_0) = \inf_{r} \Big\lbrace A^\varepsilon[\eta_\varepsilon]:\text{among all}\; \eta_\varepsilon(\cdot)\;\text{solve}\;\eqref{E-L}\;\text{with energy}\; r\; \Big\rbrace.
\end{equation}
We proceed to get different estimates for $r\leq 0$ and $r>0$. For simplicity, let us introduce the following notation. For $I$ be an interval of $\mathbb{R}$, we define $\inf_{r\in I} A^\varepsilon[\eta_\varepsilon]$ which means the infimum over all solutions $\eta_\varepsilon(\cdot)$ that solve \eqref{E-L} and with all energies $r\in I$.

\begin{prop}\label{ccc} When $r\leq 0$, we have the following estimate:
\begin{equation}\label{neg.r.fi}
\left|\inf_{r\leq 0} A^\varepsilon[\eta_\varepsilon] - u_0(x_0)\right| \leq \left(\sqrt{2\Vert V\Vert_{L^\infty}} + \Vert u_0'\Vert_{L^\infty}\right)\varepsilon.
\end{equation}
\end{prop}
Lemma \ref{key lemma} is crucial in establishing the proof of Proposition \ref{ccc}.
\begin{proof} Let $\eta_\varepsilon(\cdot)$ be a solution to \eqref{full-E-L} with $r \in \left[V(x_0,\varepsilon^{-1}x_0),0\right]$ we claim that
\begin{equation}\label{clm.neg.r}
\underline{y}_0\leq \eta_{\varepsilon}(s) \leq \overline{y}_0 \qquad\text{for all}\; s\in [0,\varepsilon^{-1}t_0],
\end{equation}
where 
\begin{align*}
\overline{y}_0 &= \min \left\lbrace\, y\in \left[\varepsilon^{-1}x_0,\varepsilon^{-1}x_0+1\right): b(y) = 0 \right\rbrace,\\
\underline{y}_0 &= \max \left\lbrace y\in \left(\varepsilon^{-1}x_0-1,\varepsilon^{-1}x_0\right]: b(y) = 0 \right\rbrace.
\end{align*}
The existence of $\underline{y}_0$ and $\overline{y}_0$ is due to the periodicity of $b(\cdot)$ and $b(y_0) = 0$. Recall that $\eta_\varepsilon(\cdot)$ satisfies the following equation thanks to the conservation of energy \eqref{H=r}:
\begin{equation*}
\begin{cases}
|\dot{\eta}_\varepsilon(s)| &= \sqrt{2\left(r-V\big(\varepsilon\eta_\varepsilon(s),\eta_\varepsilon(s)\big)\right)}, \qquad s\in (0,\varepsilon^{-1}t_0),\\
\;\eta_\varepsilon(0) &= \varepsilon^{-1}x_0.
\end{cases}
\end{equation*}
Let us define $\gamma_{+}:[0,\infty)\rightarrow \mathbb{R}$ and $\gamma_{-}:[0,\infty)\rightarrow \mathbb{R}$ such that
\begin{equation}\label{gamma_+}
\begin{cases}
\dot{\gamma}_{+}(s) &=  \;\;\sqrt{-2V\left(\varepsilon \gamma_{+}(s),\gamma_{+}(s) \right)} \quad\text{on}\quad (0,+\infty),\\
\gamma_{+}(0) &= \varepsilon^{-1}x_0,
\end{cases}
\end{equation}
and 
\begin{equation}\label{gamma_-}
\begin{cases}
\dot{\gamma}_{-}(s) &= -\sqrt{-2V\left(\varepsilon \gamma_{-}(s),\gamma_{-}(s) \right)} \quad\text{on}\quad (0,+\infty),\\
\gamma_{-}(0) &= \varepsilon^{-1}x_0,
\end{cases}
\end{equation}
respectively. To be precise, there are two cases: 
\begin{itemize}
\item $V\left(x_0,\varepsilon x_0\right) = 0$, by Lemma \ref{key lemma} we have $x\mapsto \sqrt{-V(\varepsilon x,x)}$ is Lipschitz on $\left[\varepsilon^{-1}x_0,\varepsilon^{-1}x_0+1\right]$. By uniqueness of solutions to \eqref{gamma_+} and \eqref{gamma_-} we have $\gamma_-(s)\equiv \gamma_+(s) \equiv \varepsilon^{-1}x_0$ for all $s\in [0,+\infty)$.
\item $V(x_0,\varepsilon x_0) \neq 0$, the solution $\gamma_+(\cdot)$ exists at least until $\gamma_+(\cdot)$ goes passing $\varepsilon^{-1}x_0+1$. Indeed, $\gamma_+(\cdot)$ remains staying inside $\left[\varepsilon^{-1}x_0,\varepsilon^{-1}x_0+1\right]$ and hence solution exists on $(0,+\infty)$. To see this, we first observe that $\gamma_+(\cdot)$ is increasing and for each time $t>0$, from 
\eqref{gamma_+} we have
\begin{equation*}
t = \int_{\gamma_+(0)}^{\gamma_+(t)} \frac{dx}{\sqrt{-V(\varepsilon x,x)}}.
\end{equation*}
Thus, the amount of time $\gamma_+(\cdot)$ needs to reach $\overline{y}_0$ is $\int_{\gamma_+(0)}^{\overline{y}_0} \frac{dx}{\sqrt{-V(\varepsilon x,x)}} = +\infty$ since $x\mapsto \sqrt{-V(\varepsilon x,x)}$ is Lipschitz on $\left[\varepsilon^{-1}x_0,\varepsilon^{-1}x_0+1\right]$ by Lemma \ref{key lemma}. We conclude that $\gamma_+(s)\rightarrow \overline{y}_0$ and similarly $\gamma_-(s)\rightarrow \underline{y}_0$ as $ s\rightarrow\infty$.
\end{itemize}
As a consequence, we have
\begin{equation}\label{sub.clm.neg.r}
\underline{y}_0\leq \gamma_{-}(s) \leq \eta_\varepsilon(s)\leq \gamma_{+}(s) \leq \overline{y}_0 \qquad\text{for all}\qquad s\in [0,\varepsilon^{-1}t_0]
\end{equation}
and thus \eqref{clm.neg.r} follows. Now we utilize \eqref{clm.neg.r} to estimate $A^\varepsilon[\eta_\varepsilon]$. For any $\eta_\varepsilon$ which solves \eqref{full-E-L} we have
\begin{equation}\label{neg.rest.1}
A^\varepsilon[\eta_\varepsilon]  \geq u_0\left(\varepsilon\eta_\varepsilon(\varepsilon^{-1}t_0)\right) \geq u_0\left(\varepsilon\eta_\varepsilon(0)\right) - \Vert u_0'\Vert_{L^\infty}\varepsilon.
\end{equation}
On the other hand, 
\begin{align}\label{neg.rest.2}
\inf_{r\leq 0} A^\varepsilon[\eta_\varepsilon] \leq A^\varepsilon\left[\gamma_{+}\right] &= \varepsilon\int_{\gamma_+(0)}^{\gamma_{+}(\varepsilon^{-1}t_0)} \sqrt{-2V(\varepsilon x,x)}dx + u_0\left(\varepsilon\gamma_{+}(\varepsilon^{-1}t_0)\right)\nonumber\\
&\leq u_0(\varepsilon \eta_\varepsilon(0)) + \left(\sqrt{2\Vert V\Vert_{L^\infty}} + \Vert u_0'\Vert_{L^\infty}\right)\varepsilon. 
\end{align}
thanks to \eqref{sub.clm.neg.r}. From \eqref{neg.rest.1} and \eqref{neg.rest.2} we obtain our claim \eqref{neg.r.fi}.
\end{proof}

For each $r\in \left(0,\infty\right)$, equation \eqref{full-E-L} has exactly two distinct solutions $\eta_{1,r,\varepsilon}(\cdot)$ and $\eta_{2,r,\varepsilon}(\cdot)$ thanks to the conservation of energy \eqref{H=r}. They are
\begin{equation}\label{p.1st.case}
\begin{cases}
\dot{\eta}_{\varepsilon}(s) &=\;\;\; \sqrt{2\big(r-V\left(\varepsilon \eta_{\varepsilon}(s),\eta_{\varepsilon}(s)\right)\big)} \qquad\text{on}\qquad (0,\varepsilon^{-1}t_0),\\
\eta_{\varepsilon}(0) &= \varepsilon^{-1}x_0,
\end{cases}
\end{equation}
and
\begin{equation}\label{p.2nd.case}
\begin{cases}
\dot{\eta}_{\varepsilon}(s) &= -\sqrt{2\big(r-V\left(\varepsilon \eta_{\varepsilon}(s),\eta_{\varepsilon}(s)\right)\big)} \qquad\text{on}\qquad (0,\varepsilon^{-1}t_0),\\
\eta_{\varepsilon}(0) &= \varepsilon^{-1}x_0,
\end{cases}
\end{equation}
respectively. Let us consider the first case $\eta_\varepsilon(\cdot)$ solves \eqref{p.1st.case} since the other case is similar. Since $\dot{\eta}_\varepsilon(s)> 0$ we have
\begin{equation}\label{pa.ch.var}
t_0 = \varepsilon\int_0^{\varepsilon^{-1}t_0} \frac{\dot{\eta}_\varepsilon(s)}{\dot{\eta}_\varepsilon(s)}\;ds  = \varepsilon \int_{\eta_\varepsilon(0)}^{\eta_\varepsilon(\varepsilon^{-1}t_0)} \frac{dx}{\sqrt{2(r - V(\varepsilon x,x))}}.
\end{equation}
This holds true for every $\varepsilon>0$, thus we deduce that $\eta_\varepsilon\left(\varepsilon^{-1}t_0\right)\rightarrow +\infty$ as $\varepsilon\rightarrow 0^+$. It is also clear that for all $\varepsilon>0$ then
\begin{equation}\label{bound}
t_0\sqrt{2r}\leq \varepsilon \eta_\varepsilon(\varepsilon^{-1}t_0) - x_0 \leq t_0\sqrt{2\left(r+\Vert V\Vert_{L^\infty}\right)}.
\end{equation}
By the conservation of energy \eqref{H=r} we can write the action functional as
\begin{equation}\label{A^eps}
A^\varepsilon[\eta_\varepsilon] =  rt_0 + 2\varepsilon \int_{\eta_\varepsilon(0)}^{\eta_\varepsilon(\varepsilon^{-1}t_0)} \frac{-V(\varepsilon x,x)}{\sqrt{2(r-V(\varepsilon x,x))}}\;dx + u_0\left(\varepsilon\eta_\varepsilon(\varepsilon^{-1}t_0)\right).
\end{equation}
We observe that the infimum of the optimization problem \eqref{u^e.def.e:1} should be taken over $r$ not too big. 
\begin{prop}\label{cccx} There exists $r_0>0$ depends only on $\mathrm{Lip}(u_0)$ and $\Vert V\Vert_{L^\infty}$ such that
\begin{equation}\label{u^e.def.e:1.5}
\inf_{r\geq r_0} A^\varepsilon[\eta_\varepsilon]\geq u^\varepsilon(x_0,t_0) + t_0.
\end{equation}

\end{prop}
\begin{proof} If $\eta_\varepsilon$ is a solution to \eqref{p.1st.case} with $r>0$, then from \eqref{A^eps} we have
\begin{align}\label{res.pos.r.1}
A^\varepsilon[\eta_\varepsilon] &\geq rt_0 + u_0\left(\varepsilon\eta_\varepsilon(\varepsilon^{-1}t_0)\right)\nonumber\\
&\geq   rt_0+u_0(x_0) - \Vert u_0'\Vert_{L^\infty}\left|\varepsilon \eta_\varepsilon(\varepsilon^{-1}t_0)-x_0\right| \nonumber\\
&\geq rt_0+u_0(x_0) - \Vert u_0'\Vert_{L^\infty}t_0\sqrt{2(r+\Vert V\Vert_{L^\infty})}
\end{align}
thanks to \eqref{bound}. On the other hand, by assumption (H3) we can define
\begin{equation*}
\overline{C} = \sup_{(x,y)}\Big\lbrace |H(x,y,p)|: |p|\leq \Vert u_0'\Vert_{L^\infty}\Big\rbrace < \infty
\end{equation*}
then $\overline{u}(x,t) = u_0(x)+ \overline{C} t$ is a viscosity supersolution to \eqref{C_varepsilon}, therefore
\begin{equation}\label{res.pos.r.2}
u^\varepsilon(x_0,t_0) \leq \overline{u}(x_0,t_0) = u_0(x_0) + \overline{C} t_0.
\end{equation}
There exists $r_0>0$ such that for $r\geq r_0$ we have 
\begin{align*}
r\geq \overline{C} +1 + \Vert u_0'\Vert_{L^\infty}\sqrt{2\left(r+\Vert V\Vert_{L^\infty}\right)},
\end{align*}
which is equivalent to
\begin{equation*}
rt_0+u_0(x_0) - \Vert u_0'\Vert_{L^\infty}t_0\sqrt{2\left(r+\Vert V\Vert_{L^\infty}\right)} \geq u_0(x_0) + (\overline{C} +1)t_0. 
\end{equation*}
This estimate together with \eqref{res.pos.r.1} and \eqref{res.pos.r.2} gives us 
\begin{equation*}
A^\varepsilon[\eta_\varepsilon] \geq u^\varepsilon(x_0,t_0) + t_0\quad\text{for all}\quad r\geq r_0
\end{equation*}
which proves our claim \eqref{u^e.def.e:1.5}, as the case $\eta_\varepsilon$ solves \eqref{p.2nd.case} can be done similarly.
\end{proof}
With \eqref{u^e.def.e:1.5}, the optimization problem \eqref{u^e.def.e:1} can be reduced to
\begin{equation}\label{aaa}
u^\varepsilon(x_0,t_0) =  \min \left\lbrace \inf_{r\leq 0} A^{\varepsilon}[\eta_\varepsilon], \inf_{0< r< r_0} A^{\varepsilon}[\eta_\varepsilon] \right\rbrace.
\end{equation}
Thanks to \eqref{neg.r.fi}, we only need to focus on the case $0< r < r_0$. For simplicity, let us define the following interval $I_0\subset \mathbb{R}$ to be
\begin{equation*}
I_0 = I_0(T,R) = \left[-R,c_0+R\right] \qquad\text{where}\qquad c_0 = T\sqrt{2\left(r_0+\Vert V\Vert_{L^\infty}\right)}.
\end{equation*}
Since \eqref{bound} is true for all $0<r<r_0$, for all $(x_0,t_0)\in [-R,R]\times [0,T]$ we have 
\begin{equation*}
\varepsilon\eta_\varepsilon(\varepsilon^{-1}t_0) \in I_0.
\end{equation*}
Let us define $c_{1,r}>0$ and $c_{2,r}<0$ be unique numbers such that
\begin{equation}\label{def.c_r}
\int_{x_0}^{c_{1,r}}\int_0^1 \frac{dydx}{\sqrt{2\left(r-V(x,y)\right)}} =\int_{c_{2,r}}^{x_0} \int_0^1 \frac{dydx}{\sqrt{2(r-V(x,y))}} = t_0,
\end{equation}
repsectively.
\begin{prop}\label{eps.eta(eps^-1)-c_r} Let $ \alpha_T = \min_{x\in I_0} a(x)$ and $ \beta_T = \max_{x\in I_0} a(x)$, then 
\begin{equation}\label{bound-1}
\left|\varepsilon\eta_\varepsilon(\varepsilon^{-1}t_0) - c_{1,r}\right| \leq C_K\varepsilon
\end{equation}
for $0 <r< r_0$ where $C_K$ is a constant only depends on $R,T$ and $V$.
\end{prop}
\begin{proof} Let us define $\mathcal{K}_r(x,y) = \frac{1}{\sqrt{2(r-V(x,y))}}$ for $(x,y)\in \mathbb{R}\times \mathbb{T}$. From \eqref{pa.ch.var} and \eqref{def.c_r} we have
\begin{equation}\label{June 6:1}
t_0 = \int_{x_0}^{\varepsilon
\eta_\varepsilon(\varepsilon^{-1}t_0)} \mathcal{K}_r\left(x,\frac{x}{\varepsilon}\right)dx  = \int_{x_0}^{c_{1,r}} \int_0^1 \mathcal{K}_r(x,y)\;dy\;dx. 
\end{equation}
Using Lemma \ref{Lemma on average - C1 version} we obtain
\begin{equation}\label{June 6:2}
\left|\int_{x_0}^{\varepsilon\eta_\varepsilon(\varepsilon^{-1}t_0)} \mathcal{K}_r\left(x,\frac{x}{\varepsilon}\right)\;dx- \int_{x_0}^{\varepsilon\eta_\varepsilon(\varepsilon^{-1}t_0)} \int_0^1 \mathcal{K}_r(x,y)\;dy dx\right| \leq K\varepsilon 
\end{equation}
where
\begin{equation}\label{June 6:3}
K = 2\max_{x\in I_0}\int_0^1 \mathcal{K}_r(x,y)\;dy + c_0\max_{x\in I_0}\int_0^1 \frac{\partial K_r}{\partial x}(x,y)\;dy.
\end{equation}
Using \eqref{June 6:1} in \eqref{June 6:2} we have
\begin{equation*}
 \left| \int_{c_{1,r}}^{\varepsilon\eta_\varepsilon(\varepsilon^{-1}t_0)} \int_0^1 \mathcal{K}_r(x,y)\;dy dx\right| \leq K\varepsilon
\end{equation*}
which implies that
\begin{equation}\label{June 6:4}
\left(\min_{x\in I_0} \int_{0}^1 \mathcal{K}_r(x,y)\;dy \right)\left|\varepsilon\eta_\varepsilon(\varepsilon^{-1}t_0) - c_{1,r}\right| \leq K\varepsilon.
\end{equation}
On $I_0$ we have $0<\alpha_T \leq a(x) \leq \beta_T$, which implies that
\begin{align}\label{June 6:sub-2}
 \int_0^1 \frac{dy}{\sqrt{2(r-\beta_T b(y))}} &\leq \min_{x\in I_0}\int_0^1 \mathcal{K}_r(x,y)\;dy \nonumber \\
 &\leq \max_{x\in I_0}\int_0^1 \mathcal{K}_r(x,y)\;dy \leq \int_0^1 \frac{dy}{\sqrt{2(r-\alpha_T b(y))}}.
\end{align}
Since $\alpha_T \leq \beta_T$, it is clear that
%\begin{align}\label{bound.2}
%\int_0^1 \frac{dy}{\sqrt{2(r-\alpha_T b(y))}} &= \frac{1}{\sqrt{\alpha_T}}\int_0^1 \frac{dy}{\sqrt{2\left(\left(r/\alpha_T\right) -b(y)\right)}}\nonumber\\
%&\leq \frac{1}{\sqrt{\alpha_T}}\int_0^1 \frac{dy}{\sqrt{2\left(\left(r/\beta_T\right) -b(y)\right)}} = \sqrt{\frac{\beta_T}{\alpha_T}} \int_0^1 \frac{dy}{\sqrt{2(r-\beta_T b(y))}}.
%\end{align}
\begin{align}\label{bound.2}
\int_0^1 \frac{dy}{\sqrt{2(r-\alpha_T b(y))}}  \leq \sqrt{\frac{\beta_T}{\alpha_T}} \int_0^1 \frac{dy}{\sqrt{2(r-\beta_T b(y))}}.
\end{align}
From direct calculation
%\begin{align*}
%\frac{\partial \mathcal{K}_r}{\partial x}(x,y) = \frac{a'(x)b(y)}{2(r-a(x)b(y))}\frac{1}{\sqrt{2(r-a(x)b(y))}} 
%\end{align*}
we have
\begin{equation}\label{June 6:sub-3}
\max_{x\in I_0}\int_0^1 \left|\frac{\partial \mathcal{K}_r}{\partial x}(x,y)\right|\;dy \leq \frac{1}{2}\max_{x\in I_0}\left|\frac{a'(x)}{a(x)}\right|\int_0^1\frac{dy}{\sqrt{2(r-\alpha_T b(y))}}.
\end{equation}
Use \eqref{June 6:sub-2} and \eqref{June 6:sub-3} in \eqref{June 6:3} we  deduce that
\begin{equation}\label{June 6:5}
K \leq  \left(2 + \frac{c_0}{2} \max_{x\in I_0}\left|\frac{a'(x)}{a(x)}\right|\right)\left(\int_0^1 \frac{dy}{\sqrt{2(r-\alpha_T b(y))}}\right).
\end{equation}
Next, we use \eqref{June 6:sub-2}, \eqref{June 6:5} in \eqref{June 6:4} to deduce that
\begin{equation*}\label{bound.1}
\int_0^1 \frac{dy}{\sqrt{2(r-\beta_T b(y))}}\left|\varepsilon\eta_\varepsilon(\varepsilon^{-1} t) - c_{1,r}\right| \leq \left(2 + \frac{c_0}{2}\max_{x\in I_0}\frac{|a'(x)|}{a(x)}\right)\left(\int_0^1 \frac{dy}{\sqrt{2(r-\alpha_T b(y))}}\right)\varepsilon.
\end{equation*}
From that and \eqref{bound.2} we deduce \eqref{bound-1} with
\begin{equation}\label{C_K}
C_K = \sqrt{\frac{\beta_T}{\alpha_T}}\left(2 + \frac{c_0}{2}\max_{x\in I_0}\left|\frac{a'(x)}{a(x)}\right|\right).
\end{equation}
It is clear that $C_K$ depends only on $R,T$ and $a(x)$.
\end{proof}
In view of \eqref{A^eps}, for $0 <r < r_0$ we aim to show that the integral term is close to its average with an error of order $\mathcal{O}(\varepsilon)$.

\begin{prop}\label{pro2ndterm} For $0<r<r_0$, in view of \eqref{A^eps} we have that
\begin{equation}\label{est.Lag.3}
\left|\varepsilon\int_{x_0}^{\eta_{\varepsilon}(\varepsilon^{-1}t_0)} \frac{-V(\varepsilon x,x)}{\sqrt{2(r-V(\varepsilon x,x))}}\;dx - \int_{x_0}^{c_{1,r}} \int_0^1 \frac{-V(x,y)}{\sqrt{2(r-V(x,y))}}\;dy\;dx\right| \leq C_F\varepsilon
\end{equation}
where $C_F$ is some constant only depends on $R,T$ and $V$. 
\end{prop}

%\begin{align}
%&\left|\varepsilon\int_{x_0}^{\eta_{\varepsilon}(\varepsilon^{-1}t_0)} \frac{-V(\varepsilon x,x)}{\sqrt{2(r-V(\varepsilon x,x))}}\;dx - \int_{x_0}^{c_{1,r}} \int_0^1 \frac{-V(x,y)}{\sqrt{2(r-V(x,y))}}\;dy\;dx\right| \leq C_F\varepsilon
%\end{align}
\begin{proof} To see it, let
\begin{equation*}
\mathcal{F}_r(x,y) = \frac{-V(x,y)}{\sqrt{2(r-V(x,y))}}, \qquad (x,y)\in \mathbb{R}\times \mathbb{T}.
\end{equation*}
Using Lemma \ref{Lemma on average - C1 version} we obtain
\begin{align}\label{est.Lag.1}
&\left|\int_{x_0}^{c_{1,r}} \frac{-V(x,\varepsilon^{-1}x)}{\sqrt{2(r-V(x,\varepsilon^{-1} x))}}\;dx - \int_{x_0}^{c_{1,r}} \int_0^1 \frac{-V(x,y)}{\sqrt{2(r-V(x,y))}}\;dy\;dx \right|\leq (2F_1 + c_0F_2)\varepsilon
\end{align}
where
\begin{align}
F_1&:=  \left(\Vert V\Vert_{L^\infty}\right)^{1/2} \geq \max_{\mathbb{R}\times\mathbb{T}} |\mathcal{F}_r(x,y)|  \label{F_1}\\
F_2&:= \frac{3}{2\sqrt{2}} \left(\Vert V\Vert_{L^\infty}\right)^{1/2} \max_{x\in I_0}\left|\frac{a'(x)}{a(x)}\right| \geq \max_{I_0\times \mathbb{T}}\left|\frac{\partial \mathcal{F}_r}{\partial x}(x,y)\right| . \label{F_2}
\end{align}
On the other hand, we have 
\begin{align}\label{est.Lag.2}
&\left|\int_{x_0}^{\varepsilon \eta_\varepsilon(\varepsilon^{-1}t_0)}\frac{-V(x,\varepsilon^{-1}x)}{\sqrt{2(r-V(x,\varepsilon^{-1}x))}}\;dx-\int_{x_0}^{c_
{1,r}}\frac{-V(x,\varepsilon^{-1}x)}{\sqrt{2(r-V(x,\varepsilon^{-1}x))}}\;dx \right|\nonumber\\
&\qquad\qquad\qquad\qquad\qquad\qquad\qquad\qquad\quad\; \leq F_1\left|\varepsilon\eta_{\varepsilon}(\varepsilon^{-1}t_0)-c_{1,r}\right| \leq F_1C_K \varepsilon 
\end{align}
thanks to \eqref{bound-1}. From \eqref{est.Lag.1} and \eqref{est.Lag.2} we deduce that
\begin{align}\label{est.Lag.3.1}
&\left|\int_{x_0}^{\varepsilon\eta_{\varepsilon}(\varepsilon^{-1}t_0)} \frac{-V(x,\varepsilon^{-1}x)}{\sqrt{2(r-V(x,\varepsilon^{-1} x))}}\;dx - \int_{x_0}^{c_{1,r}} \int_0^1 \frac{-V(x,y)}{\sqrt{2(r-V(x,y))}}\;dy\;dx\right| \nonumber\\
&\quad\qquad\qquad\qquad\qquad\qquad\qquad\quad\leq (2F_1 + F_2c_0+F_1C_K)\varepsilon.
\end{align}
From \eqref{est.Lag.3.1} we obtain our claim \eqref{est.Lag.3} with
\begin{equation}\label{C_F}
C_F = \left(\Vert V\Vert_{L^\infty}\right)^{1/2}\left(2+2\sqrt{\frac{\beta_T}{\alpha_T}} + c_0\left(\frac{3}{2\sqrt{2}}+\frac{1}{2}\sqrt{\frac{\beta_T}{\alpha_T}}\right)\max_{x\in {I_0}}\left|\frac{a'(x)}{a(x)}\right|\right).
\end{equation}
\end{proof}

\begin{prop}\label{prosym3} We have the following estimate:
\begin{equation}\label{negative}
\left|\inf_{{0<r< r_0}\atop {i=1,2}} A^{\varepsilon}[\eta_{i,r,\varepsilon}] - \inf_{0<r< r_0} I(r)\right| \leq C\varepsilon 
\end{equation}
where $C$ is a constant depends only on $R,T$, $a(x)$ and $\Vert V\Vert_{L^\infty}$, $I(r) = \min \{I_1(r),I_2(r)\}$ where
\begin{align}
I_1(r) &= rt_0 + 2\int_{x_0}^{c_{1,r}} \int_0^1 \frac{-V(x,y)}{\sqrt{2(r-V(x,y))}}\;dydx + u_0\left(c_{1,r}\right), \label{I_1(r).def}\\
I_2(r) &= rt_0 + \,2\int_{c_{2,r}}^{x_0}\;\, \int_0^1 \frac{-V(x,y)}{\sqrt{2(r-V(x,y))}}\;dydx + u_0(c_{2,r}). \label{I_2(r).def}
\end{align}
\end{prop}
\begin{proof} Within our notation $\eta_\varepsilon \equiv \eta_{1,r,\varepsilon} $, we have 
\begin{equation}\label{est.u0.1}
\left|u_0\left(\varepsilon \eta_\varepsilon(\varepsilon^{-1}t_0)\right) - u_0(c_{1,r})\right|\leq \Vert u_0'\Vert_{L^\infty}\left|\varepsilon \eta_\varepsilon(\varepsilon^{-1}t_0)-c_{1,r}\right|.
\end{equation}
since $u_0\in \mathrm{Lip}(\mathbb{R})$. In view of \eqref{A^eps} and \eqref{bound-1}, \eqref{est.Lag.3}, \eqref{est.u0.1} we conclude that
\begin{align}\label{fi.est.1}
\left|A^\varepsilon[\eta_\varepsilon] - I_1(r)\right| 
&\leq 2C_F \varepsilon + \Vert u_0'\Vert_{L^\infty}\left|\varepsilon \eta_\varepsilon(\varepsilon^{-1}t)-c_{1,r}\right| \nonumber\\
&\leq \big(2C_F+C_K\Vert u_0'\Vert_{L^\infty}\big)\varepsilon.
\end{align}
Taking the infimum over $0< r< r_0$ we obtain 
\begin{equation}\label{pos.r.fi.1}
\left|\inf_{0< r< r_0} A^\varepsilon [\eta_{1,r,\varepsilon}] - \inf_{0 < r < r_0} I_1(r)\right| \leq C_1 \varepsilon
\end{equation}
where $C_1 = 2C_F+C_K\Vert u_0'\Vert_{L^\infty}$. Similarly for the case $\eta_{2,r,\varepsilon}$ solves \eqref{p.2nd.case}, we obtain
\begin{equation}\label{pos.r.fi.2}
\left|\inf_{0< r< r_0} A^\varepsilon[\eta_{2,r,\varepsilon}] - \inf_{0< r< r_0} I_2(r) \right|\leq C_2 \varepsilon
\end{equation}
where $C_2$ is some constant depends on $R,T$, $a(x)$ and $\Vert V\Vert_{L^\infty}$ in the same manner as $C_1$. Thus our claim \eqref{negative} is correct with $C = \max\{C_1,C_2\}$.
\end{proof}
From \eqref{neg.r.fi}, \eqref{aaa} and \eqref{negative} we conclude that
\begin{equation*}
\left|u^\varepsilon(x_0,t_0) - u(x_0,t_0)\right| \leq \left(\max\left\lbrace \sqrt{2\Vert V\Vert_{L^\infty}} + \Vert u'_0\Vert_{L^\infty},C \right\rbrace\right)\varepsilon
\end{equation*}
and the proof is complete.
\end{proof}

\begin{cor} We have the following representation formula
\begin{equation*}
u(x_0,t_0) = \min\left \lbrace u_0(x_0), \min\left\lbrace \inf_{0< r< r_0} I_1(r),\inf_{0< r< r_0} I_2(r)\right\rbrace \right\rbrace
\end{equation*}
where $I_1(r)$ and $I_2(r)$ are defined in \eqref{I_1(r).def} and \eqref{I_2(r).def} respectively.
\end{cor}

\begin{rem}\label{remark4} If $V(x,y) = V(y)$ is independent of $x$, then the constants $C_K$ in \eqref{C_F} and $C_F$ in \eqref{C_F} are independent of $R$ and $T$. Therefore the convergence is uniform in $\mathbb{R}\times [0,\infty)$ and by carefully keeping track of all constants, we get
\begin{equation*}
C = 2\left(\Vert u_0'\Vert_{L^\infty(\mathbb{R})} + 4\sqrt{\max_{y\in \mathbb{T}} |V(y)|}\right).
\end{equation*}
Also Proposition \ref{ccc} is no longer needed in this case. 
\end{rem}

We provide a proof for Lemma \ref{Lemma on average - C1 version} (see \cite{Neukamm2017}) with an explicit bound in Appendix. This lemma is a quantitative version of the ergodic Theorem for periodic functions in one dimension. This is a generalized version of Lemma 4.2 in \cite{Mitake2018}.

\begin{lem}\label{Lemma on average - C1 version} If $F(x,y)\in \mathrm{C}^1(\mathbb{R}\times\mathbb{T})$ then for any real numbers $a<b$ we have
\begin{equation*}
\left|\int_a^b F\left(x,\frac{x}{\varepsilon}\right)\;dx - \int_a^b \left(\int_0^1 F(x,y)\;dy\right)\;dx\right| \leq C\varepsilon
\end{equation*}
where
\begin{align*}
C = 2\max_{x\in [a,b]} \int_0^1 |F(x,y)|\;dy + (b-a)\max_{x\in [a,b]} \int_0^1 \left|\frac{\partial F}{\partial x}(x,y)\right|\;dy.
\end{align*}
\end{lem}

%We give a proof for the technical Lemma \ref{key lemma} in the more general form, which is useful later for a general class of Hamiltonians in Section 3.

The following lemma is crucial in handling the minimizer paths that correspond to nonpositive energies. We provide the proof of this lemma in Appendix.

% \begin{lem}\label{key lemma} Let $\mathcal{V}\in \mathrm{C}^2\big(\R,[0,\infty)\big)$ with $\min_{x\in [0,1]}\mathcal{V}(x) = 0$. There exists a constant $L>0$ such that $|\mathcal{V}\,'(x)|\leq L \sqrt{\mathcal{V}(x)}$ for all $x\in \R$. As a consequence, $x\mapsto \sqrt{\mathcal{V}(x)}$ is Lipschitz in $\R$.
% \end{lem}

\begin{lem}\label{key lemma} Let $\mathcal{V}\in \mathrm{C}^2\big(\R,[0,\infty)\big)$ with $\min_{x\in \R}\mathcal{V}(x) = 0$. There exists a constant $L>0$ such that $|\mathcal{V}\,'(x)|\leq L \sqrt{\mathcal{V}(x)}$ for all $x\in \R$. As a consequence, $x\mapsto \sqrt{\mathcal{V}(x)}$ is Lipschitz in $\R$.
\end{lem}

\section{General strictly convex Hamiltonians setting}

\subsection{Setting and simplifications} 
Similarly to the proof of Theorem \ref{simple case}, we can assume $C_0=0$ and $V\in \mathrm{C}^2(\mathbb{R}\times \mathbb{T})$. We have the following estimate (\cite{Le2017}):
\begin{equation}\label{grad.bound}
\Vert u^\varepsilon_t\Vert_{L^\infty} + \Vert Du^\varepsilon\Vert_{L^\infty}\leq M
\end{equation}
in the viscosity sense for all $\varepsilon>0$. Since values of $H(p)$ for $|p| > M$ are irrelevant. This fact together with $H(0) = H'(0) = 0$ allows us to assume that
\begin{equation}\label{quad.grth.H}
\max\left\lbrace \frac{|p|^2}{2} - K_0, \frac{|p|^2}{2} - K_0 |p|\right\rbrace \leq H(p) \leq \min\left\lbrace \frac{|p|^2}{2} + K_0, \frac{|p|^2}{2} + K_0|p|  \right\rbrace
\end{equation}
for all $p\in \mathbb{R}$ and for some $K_0 > 0$.  Let $L(v) = \sup_{p\in \mathbb{R}} \Big(p\cdot v - H(p)\Big)$ for $v\in \mathbb{R}^n$ be the Legendre transform of $H$, then $L$ is $\mathrm{C}^2$ and strictly convex, $L(v) > L(0) = 0$ for $v\neq 0$ as well as $L(0) = L'(0) = 0$, and
\begin{equation}\label{quad.grth.L}
\max\left\lbrace \frac{|v|^2}{2} - K_0, \frac{|v|^2}{2}-K_0|v|\right\rbrace \leq L(v) \leq \min \left\lbrace \frac{|v|^2}{2} + K_0, \frac{|v|^2}{2} + K_0|v|\right\rbrace
\end{equation}
for $v\in \mathbb{R}^n$. Denote:
\begin{equation*}
\begin{cases}
H^{-1}_1&:= \left(H|_{[0,\infty)}\right)^{-1}\\
(L'_1)^{-1}&:= \left(L'|_{[0,\infty)}\right)^{-1}\\
\tilde{G}_1&:=(L'_1)^{-1}\circ H_1^{-1}
\end{cases} \qquad\text{and}\qquad \begin{cases}
H_2^{-1}&: = \left(H|_{0,+\infty)}\right)^{-1}\\
(L'_2)^{-1}&:= \left(L'|_{(-\infty,0]}\right)^{-1}\\
\tilde{G}_2&:=(L'_2)^{-1}\circ H_2^{-1}.
\end{cases}
\end{equation*}
We have $H_i' = (L_i)^{-1}$ and thus $\tilde{G}_i \equiv G_i$ for $i=1,2$ where $G_i$ are defined in the statement of Theorem \ref{general case}. We see that $x\mapsto |G_i(x)|$ is increasing on $[0,\infty)$, $x\mapsto (L'_i)^{-1}(x)$ is increasing for $i=1,2$ and
%\begin{equation*}
%\begin{cases}
%\displaystyle\frac{p}{2}-K_0\leq (L_1')^{-1}(p) \leq 2(K_0+p) &\qquad\text{if}\qquad p\geq 0,\vspace*{0.2cm}\\
%\displaystyle\frac{p}{2}-K_0\geq (L_2')^{-1}(p) \geq 2(K_0+p) &\qquad\text{if}\qquad p\leq 0.
%\end{cases}
%\end{equation*}
%Here we used the fact that $L_i^{-1}(p) = v = H_i'(p)$ for $i=1,2$. From \eqref{quad.grth.H} we obtain
%\begin{equation*}
%\begin{cases}
%\;\;\,\sqrt{2(x-K_0)}\leq H_1^{-1}(x) \leq \;\;\, \sqrt{2(x+K_0)},\\
%-\sqrt{2(x+K_0)}\leq H_2^{-1}(x) \leq -\sqrt{2(x-K_0)}
%\end{cases} \qquad\text{for}\qquad x\geq K_0.
%\end{equation*}
for all $x\geq K_0$ then
\begin{equation}\label{G_i.rate}
\begin{cases}
\displaystyle \;\;\; (1/\sqrt{2}) \sqrt{x-K_0}\leq G_1(x) \leq 2K_0 + 2\sqrt{2(x+K_0)},\vspace*{0.2cm}\\
\displaystyle -(1/\sqrt{2})\sqrt{x-K_0} \geq G_2(x) \geq 2K_0 - 2\sqrt{2(x+K_0)}.
\end{cases}
\end{equation}
As a consequence, we have $|G_i(x)| \rightarrow +\infty$ as $x\rightarrow \infty$ for $i=1,2$.
\subsection{Sketch of the proof of Theorem \ref{general case}}

For $\varepsilon>0$ and $R,T>0$, let us fix $(x_0,t_0)\in [-R,R]\times [0,T]$. Thanks to the optimal control formula we have
\begin{equation}\label{ge.u^eps.def}
u^\varepsilon(x_0,t_0)= \inf_{\eta(\cdot)\in \mathcal{T}} \left\lbrace \varepsilon\int_0^{\varepsilon^{-1}t_0} \Big(L\left(\dot{\eta}(s)\right) - V\left(\varepsilon \eta(s),\eta(s)\right)\Big)ds + u_0\left(\varepsilon\eta(\varepsilon^{-1}t_0)\right)\right\rbrace
\end{equation}
where $\mathcal{T} = \left\lbrace \eta(\cdot)\in \mathrm{AC}\left(\left[0,\varepsilon^{-1}t_0\right]\right), \varepsilon\eta(0) = x_0\right\rbrace$. For each mininmizer $\eta_\varepsilon(\cdot)\in \mathcal{T}$ to \eqref{ge.u^eps.def}, there exists $r = r(\eta_\varepsilon) \in \left[V(0,0),+\infty\right)$ such that
\begin{equation}\label{ge.H=r}
H\big(L'(\dot{\eta}_\varepsilon(s))\big) + V\left(\varepsilon\eta_\varepsilon(s),\eta_\varepsilon(s)\right) = r
\end{equation}
for all $s\in(0,\varepsilon^{-1}t_0)$. For $r\in \left[V(0,0),\infty\right)$ we have the Euler--Lagrange equation 
\begin{equation}\label{ge.full-E-L}
\begin{cases}
L''\big(\dot{\eta}_\varepsilon(s)\big)\ddot{\eta}_\varepsilon(s) = -\nabla V\big(\varepsilon\eta_\varepsilon(s),\eta_\varepsilon(s)\big)\cdot(\varepsilon,1)\qquad\text{on}\qquad \left(0,\varepsilon^{-1}t_0\right), \\
\qquad\qquad\, \dot{\eta}_\varepsilon(0) = G_i\big(r-V(x_0,\varepsilon^{-1}x_0)\big),\\
\qquad\qquad\,\eta_\varepsilon(0) =  \varepsilon^{-1} x_0.
\end{cases} \tag{E-L}
\end{equation}
where $i=1,2$. For simplicity, let us define the following action functional
\begin{equation*}
A^\varepsilon[\eta] = \varepsilon\int_0^{\varepsilon^{-1}t_0} \Big(L\left(\dot{\eta}(s)\right) - V\left(\varepsilon \eta(s),\eta(s)\right)\Big)ds + u_0\left(\varepsilon\eta(\varepsilon^{-1}t_0)\right)
\end{equation*}
for $\eta(\cdot)\in \mathcal{T}$. Thanks to \eqref{ge.H=r}, the optimization problem \eqref{ge.u^eps.def} is equivalent to
\begin{equation}\label{ge.u^e.def.e:1}
u^\varepsilon(x_0,t_0) = \inf_{r} \Big\lbrace A^\varepsilon[\eta_\varepsilon]:\text{among all}\; \eta_\varepsilon(\cdot)\;\text{solve}\;\eqref{ge.full-E-L}\;\text{with energy}\; r\; \Big\rbrace.
\end{equation}
For an interval $I\subset\mathbb{R}$ we denote by $\inf_{r\in I} A^\varepsilon[\eta_\varepsilon]$ the infimum over all solutions $\eta_\varepsilon(\cdot)$ that solve \eqref{ge.full-E-L} and with all energies $r\in I$. We proceed as follows:

\begin{enumerate}
\item[1.] Estimate for $r\leq 0$ with rate $\mathcal{O}(\varepsilon)$ (Proposition \ref{prop r<0}). 
\item[2.] There is $r_0>0$ such that we can ignore $r\geq r_0$ in \eqref{ge.u^e.def.e:1} (Proposition \ref{ignore}).
\item[3.] For $0<r<r_0$, $A^\varepsilon[\eta_\varepsilon]$ can be written as in \eqref{ge.A^eps.formula}, then we proceed to get estimates for each individual term by using an quantitative ergodic theorem (Propositions \ref{asymp0}, \ref{asymp1} and \ref{asymp3}).
%\item[4.] The result follows from these above steps.
\end{enumerate}

\begin{prop}\label{prop r<0} If $r\leq 0$ then 
\begin{equation}\label{ge.neg.r.fi}
\left|\inf_{r\leq 0} A^\varepsilon[\eta_\varepsilon] - u_0(x_0)\right| \leq \Big(H_1^{-1}\left(\Vert V\Vert_{L^\infty}\right)+ \Vert u_0'\Vert_{L^\infty}\Big)\varepsilon.
\end{equation}
\end{prop}

\begin{proof}[Sketch of the proof] The proof is similar to Proposition \ref{ccc} where the crucial Lemma \ref{key lemma} is replaced with Lemma \ref{key lemma 2}.
\end{proof}

For each $r\in \left(0,\infty\right)$, \eqref{ge.full-E-L} has exactly two distinct solutions $\eta_{1,r,\varepsilon}(\cdot)$ and $\eta_{2,r,\varepsilon}(\cdot)$ thanks to the conservation of energy \eqref{ge.H=r}. They are
\begin{equation}\label{ge.p.1st.case}
\begin{cases}
\dot{\eta}_{\varepsilon}(s) &= G_i\big(r-V\big(\varepsilon \eta_{\varepsilon}(s),\eta_{\varepsilon}(s)\big)\big) \qquad\text{on}\qquad (0,\varepsilon^{-1}t_0),\\\eta_{\varepsilon}(0) &= \varepsilon^{-1}x_0,
\end{cases}
\end{equation}
for $i=1,2$ respectively.
%\begin{equation}\label{ge.p.2nd.case}
%\begin{cases}
%\dot{\eta}_{\varepsilon}(s) &= G_2\Big(r-V\big(\varepsilon \eta_{\varepsilon}(s),\eta_{\varepsilon}(s)\big)\Big) \qquad\text{on}\qquad (0,\varepsilon^{-1}t_0),\\
%\eta_{\varepsilon}(0) &= \varepsilon^{-1}x_0,
%\end{cases}
%\end{equation}
%respectively. 
Let us consider the first case $\eta_\varepsilon(\cdot)$ solves \eqref{ge.p.1st.case} with $i=1$ since the other case is similar. Since $\dot{\eta}_{\varepsilon}(s)> 0$ for all $s\geq 0$, we have
\begin{equation}\label{ge.pa.ch.var}
t_0 = \varepsilon\int_0^{\varepsilon^{-1}t_0} \frac{\dot{\eta}_{\varepsilon}(s)}{\dot{\eta}_{\varepsilon}(s)}\;ds  = \varepsilon \int_{\eta_\varepsilon(0)}^{\eta_{\varepsilon}(\varepsilon^{-1}t_0)} \frac{dx}{G_1\big(r-V(\varepsilon x,x)\big)}.
\end{equation}
Let $\varepsilon\rightarrow 0$ we deduce that $\eta_{\varepsilon}\left(\varepsilon^{-1}t_0\right)\rightarrow +\infty$. It is also clear from \eqref{ge.p.1st.case} that
\begin{equation}\label{ge.bound}
t_0G_1(r)\leq \varepsilon \eta_{\varepsilon}(\varepsilon^{-1}t_0) - \varepsilon \eta_\varepsilon(0) \leq t_0G_1\big(r+\max|V|\big).
\end{equation}

%For a fixed $\varepsilon>0$, we observe that the infimum of the optimization problem \eqref{ge.u^e.def.e:1} should be taken over $r$ not too big.
\begin{prop}\label{ignore} There exists $r_0>0$ depends on $\mathrm{Lip}(u_0)$ and $H(p)$ such that
\begin{equation}\label{ge.u^e.def.e:1.5}
\inf_{r\geq r_0} A^\varepsilon[\eta_\varepsilon] = \inf_{r\geq r_0}\Big\lbrace A^\varepsilon[\eta_{1,r,\varepsilon}], A^\varepsilon[\eta_{2,r,\varepsilon}]\Big\rbrace \geq u^\varepsilon(x_0,t_0) + t_0.
\end{equation}
\end{prop}

\begin{proof}[Sketch of the proof] The proof is similar to Proposition \ref{cccx} where we utilize the fact that $G_1$ is increasing and satisfies \eqref{G_i.rate}.
\end{proof}

With \eqref{ge.u^e.def.e:1.5}, the optimization problem \eqref{ge.u^e.def.e:1} can be reduced to
\begin{equation}\label{ge.aaa}
u^\varepsilon(x_0,t_0) = \min \left\lbrace \inf_{r\leq 0} A^\varepsilon[\eta_\varepsilon],\inf_{0<r<r_0} A^\varepsilon[\eta_{1,r,\varepsilon}],\inf_{0<r<r_0} A^\varepsilon[\eta_{2,r,\varepsilon}] \right\rbrace.
\end{equation}
Let $\eta_\varepsilon = \eta_{1,r,\varepsilon}$, we have $L\left(\dot{\eta}_{\varepsilon}(s)\right) -V\big(\varepsilon\eta_{\varepsilon}(s),\eta_{\varepsilon}(s)\big) = -r + \dot{\eta}_{\varepsilon}(s) L'\big(\dot{\eta}_{\varepsilon}(s)\big)$. From that and \eqref{ge.p.1st.case} we can rewrite the action functional as 
\begin{equation}\label{ge.A^eps.formula}
A^\varepsilon[\eta_{\varepsilon}] = -rt_0 + \varepsilon\int_{\eta_\varepsilon(0)}^{\eta_\varepsilon(\varepsilon^{-1}t_0)} H_1^{-1}\big(r-V\big(\varepsilon x,x\big)\big)dx + u_0\big(\varepsilon\eta_{\varepsilon}(\varepsilon^{-1}t_0)\big).
\end{equation}
Define $I_0 = I_0(T,R) = [-R,c_0+R]$ where $c_0 = TG_1(r_0+\Vert V\Vert_{L^\infty})$. Since \eqref{ge.bound} is true for all $0<r<r_0$ and $(x_0,t_0)\in [-R,R]\times [0,T]$ we have $\varepsilon \eta_\varepsilon(\varepsilon^{-1}t_0) \in I_0$. Let $c_{1,r}$ and $c_{2,r}$ be unique numbers such that
\begin{equation}\label{ge.def.c_r}
\int_{x_0}^{c_{1,r}}\int_0^1 \frac{dydx}{G_1(r-V(x,y))} =\int_{c_{2,r}}^{x_0}\int_0^1 \frac{dydx}{G_2(r-V(x,y))} = t_0.
\end{equation}

\begin{prop}\label{asymp1} For $0<r< r_0$ we have
\begin{equation}\label{ge.bound-1}
\left|\varepsilon\eta_\varepsilon(\varepsilon^{-1}t_0) - c_{1,r}\right| \leq C_K\varepsilon.
\end{equation}
where $C_K = C_K(R,T,H,V)$ is a constant independent of $r$.
\end{prop}
\begin{proof}[Sketch of the proof] Let $\mathcal{K}_r(x,y) = \dfrac{1}{G_1(r-V(x,y))}$ for $(x,y)\in \mathbb{R}\times \mathbb{T}$. Similarly to proof of Proposition \ref{eps.eta(eps^-1)-c_r}, we obtain 
\begin{equation*}
C_K = 2\left(1+2c_0\tilde{K} \right)\sup_{0<r< r_0}\left(\frac{\displaystyle\max_{x\in I_0}\int_0^1 \mathcal{K}_r(x,y)\;dy }{\displaystyle\min_{x\in I_0}\int_0^1 \mathcal{K}_r(x,y)\;dy }\right)  <\infty
\end{equation*}
by assumption (A4) and 
\begin{equation*}
\tilde{K} = \sup_{0<r<r_0} \left \lbrace |V_x(x,y)|.\left|\frac{G_1'(r-V(x,y))}{G_1(r-V(x,y))}\right|: (x,y)\in I_0\times \mathbb{T}\right\rbrace < \infty
\end{equation*}
by assumption (A2).
\end{proof}

\begin{prop}\label{asymp0} For $0<r<r_0$, in view of \eqref{ge.A^eps.formula} we have
\begin{equation*}
\left|\int_{x_0}^{\varepsilon \eta_\varepsilon(\varepsilon^{-1}t_0)} H_1^{-1}\left(r-V(x,\varepsilon^{-1}x)\right)dx  - \int_{x_0}^{c_{1,r}}\int_0^1 H_1^{-1}\left(r-V(x,y)\right)dydx \right| \leq C_F \varepsilon
\end{equation*}
where $C_F$ is a constant independent of $r$. 
\end{prop}

\begin{proof}[Sketch of the proof] Define $\mathcal{F}_r(x,y) = H_1^{-1}\left(r-V(x,y)\right)$ for $(x,y)\in \mathbb{R}\times\mathbb{T}$. The proof is similar to Proposition \ref{pro2ndterm}. We use (A3) to get the bound $F_2$:
\begin{align*}
F_1&:= H_1^{-1}\left(r_0+ \Vert V\Vert_{L^\infty}\right)  \geq \max_{x\in I_0}\int_0^1 \mathcal{F}_r(x,y)\;dy\\
F_2&:= \sup \left\lbrace \frac{|V_x(x,y)|}{|G_1(-V(x,y))|}\;\Big|\; (x,y) \in I_0\times \mathbb{T}\right\rbrace \geq  \max_{x\in I_0} \int_0^1 \left|\frac{\partial F_r}{\partial x}(x,y)\right|dy.
\end{align*}
Similar to Proposition \ref{pro2ndterm}, we can compute $C_F$ as $C_F = 2F_1+c_0F_2 + C_KF_1$.
\end{proof}

\begin{prop}\label{asymp3} We have the following estimate:
\begin{equation}\label{ge.pos.r.fi}
\left|\inf_{{0< r < r_0}\atop {i=1,2}} A^\varepsilon [\eta_{i,r,\varepsilon}] - \inf_{0< r < r_0} I(r)\right| \leq C \varepsilon
\end{equation}
where $C$ is a constant independent of $r$ and $I(r) = \min \left\lbrace I_1(r),I_2(r)\right\rbrace$ where
\begin{align}
I_1(r) &= -rt_0 + \int_{x_0}^{c_{1,r}} \int_0^1 H^{-1}_1\left(r-V\left(x,y\right)\right)dy\;dx + u_0\left(c_{1,r}\right),\label{ge.I1(r).def}\\
I_2(r) &= -rt_0 + \int_{c_{2,r}}^{x_0}\; \int_0^1 H^{-1}_2\left(r-V\left(x,y\right)\right)dy\;dx + u_0\left(c_{2,r}\right).\label{ge.I2(r).def}
\end{align}
\end{prop}

The proof is omitted since it is similar to Proposition \ref{prosym3}.

Finally, using \eqref{ge.neg.r.fi} and \eqref{ge.pos.r.fi} in \eqref{ge.aaa} we obtain the claim of Theorem \ref{general case}.
\begin{equation*}
\left|u^\varepsilon(x_0,t_0) - u(x_0,t_0)\right| \leq \left(\max\Big\lbrace H_1^{-1}\left(\Vert V\Vert_{L^\infty}\right) + \Vert u'_0\Vert_{L^\infty},C \Big\rbrace\right)\varepsilon.
\end{equation*}

\begin{lem}\label{key lemma 2} Let $\mathcal{V}\in \mathrm{C}^2\big([0,1],[0,\infty)\big)$ with $\min_{x\in [0,1]}\mathcal{V}(x) = 0$ and $\mathcal{V}(0) = \mathcal{V}(1)$.
\begin{itemize}
\item[(i)] Let $H, G_1,G_2$ be defined as in Theorem \ref{general case}. If 
\begin{equation}\label{lemA1}
\limsup_{p\rightarrow 0} \left|\frac{H''(p)}{H'(p)}\sqrt{H(p)}\right| < \infty,
\end{equation}
then $x\mapsto G_i(\mathcal{V}(x))$ is Lipschitz on $[0,1]$ for $i=1,2$.
\item[(ii)] If $H$, defined in Theorem \ref{general case}, satisfies $H''(0) > 0$ then we have something stronger than \eqref{lemA1}
\begin{equation}\label{lemA1_pt1}
\limsup_{p\rightarrow 0} \left|\frac{\sqrt{H(p)}}{H'(p)}\right| < \infty.
\end{equation}
In this case we have further that $C_{1,I}\sqrt{x}\leq|G_i(x)|\leq C_{2,I}\sqrt{x}$ on any bounded subset $I\subset\mathbb{R}$, where $i=1,2$ and $C_{I,1},C_{I,2}>0$.
\item[(iii)] If $H$, defined in Theorem \ref{general case}, satisfies $H\in \mathrm{C}^3(\mathbb{R})$ then 
\begin{equation}\label{lemA1_pt2}
\limsup_{p\rightarrow 0} \frac{\left|H''(p)\right|}{\sqrt{|H'(p)|}} < \infty.
\end{equation}
As a consequence, we have something stronger than \eqref{lemA1}
\begin{equation}\label{lemA1_pt3}
\limsup_{p\rightarrow 0} \left|\frac{H''(p)}{H'(p)}\sqrt{\frac{H(p)}{|p|}}\right| < \infty.
\end{equation}
\item[(iv)] If $H(p) = |p|^\gamma$ where $\gamma\geq 2$ then \eqref{lemA1} holds true.
\end{itemize}
\end{lem}

\begin{cor} We have the following representation formula
\begin{equation*}
u(x_0,t_0) = \min\left \lbrace u_0(x_0), \min\left\lbrace \inf_{0< r< r_0} I_1(r),\inf_{0< r< r_0} I_2(r)\right\rbrace \right\rbrace.
\end{equation*}
where $I_1(r)$ and $I_2(r)$ are defined in \eqref{ge.I1(r).def} and \eqref{ge.I2(r).def}, respectively.
\end{cor}

\begin{proof}[Proof of Corollary \ref{corollary simplified}] In order to apply Theorem \ref{general case} we need to check conditions (A0),(A2),(A3),(A4). Let us fix a compact interval $I\subset\mathbb{R}$, in the assumption of $V$ let us denote $\alpha,\beta,f$ by $\alpha_I,\beta_I,f_I$ for simplicity.

If $H(p) = |p|^\gamma$ where $\gamma
 \geq 2$ then $|G_i(p)| = \gamma |p|^{1-\frac{1}{\gamma}}$ and $|G_i'(p)|=(\gamma-1)|p|^{-\frac{1}{\gamma}}$. Therefore conditions (A0),(A2),(A3) follow from direct computation. (A4) follows since $p\mapsto |G_i(p)|$ is increasing and for any compact interval $I\subset\mathbb{R}$ then
 \begin{equation*}
\max_{x\in I}\int_0^1 \frac{dy}{|G_i(r-V(x,y))|}  \leq \frac{1}{\gamma}\left(\frac{\beta}{\alpha}\right)^{1-\frac{1}{\gamma}} \min_{x\in I}\int_0^1 \frac{dy}{|G_i(r-V(x,y)|}.
 \end{equation*}

 %%%%%%%%%%%%%%%%%%%%%%%%%%%%%%%%%%%%%%
%%%%%%%%%%%%%%%%%%%%%%%%%%%%%%%%%%%%%%
%%%%%%%%%%%%%%%%%%%%%%%%%%%%%%%%%%%%%%
%%%%%%%%%%%%%%%%%%%%%%%%%%%%%%%%%%%%%%
\begin{comment}
\begin{align*}
\max_{x\in I}\int_0^1 \frac{dy}{|G_i(r-V(x,y))|} &\leq \int_0^1 \frac{dy}{|G_i(r+\alpha f(y))|}\\
&\leq \int_0^1 \frac{dy}{\left|G_i\left(\alpha/\beta\left(r+\beta f(y)\right)\right)\right|}  \\
&= \frac{1}{\gamma}\left(\frac{\beta}{\alpha}\right)^{1-\frac{1}{\gamma}}\int_0^1 \frac{dy}{|G_i(r+\beta f(y))|} \\
&\leq \frac{1}{\gamma}\left(\frac{\beta}{\alpha}\right)^{1-\frac{1}{\gamma}} \min_{x\in I}\int_0^1 \frac{dy}{|G_i(r-V(x,y)|}.
\end{align*}

\end{comment}
%%%%%%%%%%%%%%%%%%%%%%%%%%%%%%%%%%%%%%
%%%%%%%%%%%%%%%%%%%%%%%%%%%%%%%%%%%%%%
%%%%%%%%%%%%%%%%%%%%%%%%%%%%%%%%%%%%%%
%%%%%%%%%%%%%%%%%%%%%%%%%%%%%%%%%%%%%%

In general when $H''(0) > 0$, condition (A0) follows from Lemma \ref{key lemma 2}. On the bounded set $\left[0,\Vert V\Vert_{L^\infty}+1\right]$ by Lemma \ref{key lemma 2} we have $C_1\sqrt{x} \leq |G_i(x)|\leq C_2\sqrt{x}$ for $i=1,2$ and for some $C_1,C_2>0$. For $i=1,2$, $0<r<1$ and $x\in I$ we have
\begin{align*}
|V_x(x,y)|.\frac{|G_i'(r-V(x,y))|}{|G_i(r-V(x,y))|}\leq \frac{|V_x(x,y)|}{|V(x,y)|} \left(\frac{\sqrt{H(\xi)}}{|G_i(H(\xi))|}\right)\left( \frac{|H''(\xi)|}{|H'(\xi)|}\sqrt{H(\xi)}\right)
\end{align*}
where $\xi = H_i^{-1}(r-V(x,y))$. The right hand side is bounded as $r\rightarrow 0^+$ due to (A0), $C_1\sqrt{x} \leq |G_i(x)|\leq C_2\sqrt{x}$ and \eqref{rmk.3.2}, thus (A2) follows. Condition (A3) is true since for $x\in I$ then
\begin{align*}
\left|\frac{V_x(x,y)}{G_i(V(x,y))}\right| \leq  \left|\frac{V_x(x,y)}{V(x,y)}\right|.\frac{\sqrt{|V(x,y)|}}{|G_i(|(V(x,y)|)|}.\sqrt{|V(x,y)|}.
\end{align*}
%The right hand side is bounded as well due to \eqref{rmk.3.2}, \eqref{coro1} and the fact that $V$ bounded. 
Finally, for $i=1,2$ then $x\mapsto |G_i(x)|$ is increasing, using  $C_1\sqrt{x} \leq |G_i(x)|\leq C_2\sqrt{x}$ we deduce that for $0<r<1$ then
\begin{align*}
\max_{x\in I}\int_0^1 \frac{dy}{|G_i(r-V(x,y))|} &\leq \int_0^1 \frac{dy}{|G_i(r+\alpha f(y))|}\leq \int_0^1 \frac{dy}{C_1\sqrt{r+\alpha f(y)}}\\
\min_{x\in I}\int_0^1 \frac{dy}{|G_i(r-V(x,y))|} &\geq \int_0^1 \frac{dy}{|G_i(r+\beta f(y))|} \geq \int_0^1 \frac{dy}{C_2\sqrt{r+\beta f(y)}}.
\end{align*}
Since $\alpha \leq \beta$, we have $\sqrt{r+\alpha f(y)} \geq \sqrt{\frac{\alpha}{\beta}\left(r+\beta f(y)\right)}$ and therefore
\begin{equation*}
\int_0^1 \frac{dy}{C_1\sqrt{r+\alpha f(y)}} \leq \left(\frac{C_2}{C_1}\sqrt{\frac{\beta}{\alpha}}\right) \int_0^1 \frac{dy}{C_2\sqrt{r+\beta f(y)}} 
\end{equation*}
and thus (A4) follows.
\end{proof}

\section{Appendix}

\begin{proof}[Proof of Lemma \ref{Lemma on average - C1 version}] Since $y\mapsto F(x,y)$ is periodic, we have $y\mapsto \frac{\partial F}{\partial x} (x,y)$ is also periodic. Let us define
\begin{equation*}
G(x,y) = \int_0^y \left(F(x,z) - \int_0^1 F(x,\zeta)\;d\zeta\right) \;dz 
\end{equation*}
then $\frac{\partial G}{\partial y} (x,y) = F(x,y) - \int_0^1 F(x,\zeta)\;d\zeta$. Since $G$ is periodic in $y$, $\frac{\partial G}{\partial x}$ is also periodic in $y$. %because
%\begin{align*}
%\frac{\partial G}{\partial x}(x,y+1)  &= \int_0^{y+1} \left(\frac{\partial F}{\partial x}(x,z) - \int_0^1 \frac{\partial F}{\partial x}(x,\zeta)\;d\zeta\right) \;dz\\
%&= \int_0^{y} \left(\frac{\partial F}{\partial x}(x,z+1) - \int_0^1 \frac{\partial F}{\partial x}(x,\zeta)\;d\zeta\right) \;dz = \frac{\partial G}{\partial x}(x,y).
%\end{align*}
Thus $G$ and $\frac{\partial G}{\partial x}$ are bounded in $y$. The fact that $\frac{\partial F}{\partial x}$ is bounded in $x$ implies $\frac{\partial G}{\partial x}$ is bounded in $x$ as well. Let $g_\varepsilon(x) = \varepsilon G\left(x,\frac{x}{\varepsilon}\right)$ we obtain
\begin{align*}
\frac{d}{dx} \Big(g_\varepsilon(x)\Big) = \varepsilon \frac{\partial G}{\partial x}\left(x,\frac{x}{\varepsilon}\right) + \frac{\partial G}{\partial y}\left(x,\frac{x}{\varepsilon}\right) = \varepsilon \frac{\partial G}{\partial x}\left(x,\frac{x}{\varepsilon}\right) + F\left(x,\frac{x}{\varepsilon}\right) - \int_0^1 F(x,\zeta)\;d\zeta.
\end{align*}
Thus
\begin{align*}
\int_a^b F\left(x,\frac{x}{\varepsilon}\right)dx - \int_a^b \int_0^1 F(x,\zeta)\;d\zeta dx %&= \int_a^b \left(\frac{d}{dx} \Big(g_\varepsilon(x)\Big) - \varepsilon \frac{\partial G}{\partial x}\left(x,\frac{x}{\varepsilon}\right)\right)\;dx\\
&= \varepsilon \left[G\left(b,\frac{b}{\varepsilon}\right) - G\left(a,\frac{a}{\varepsilon}\right) - \int_a^b \frac{\partial G}{\partial x}\left(x,\frac{x}{\varepsilon}\right)dx\right].
\end{align*}
Note that by the way we defined $G$, we also have 
\begin{align*}
\max_{(x,y)} |G(x,y)| &\leq \max_{x\in [a,b]} \int_0^1 \left|F(x,y)\right|\;dy\\
\left|\int_a^b \frac{\partial G}{\partial x}\left(x,\frac{x}{\varepsilon}\right)\;dx\right| &\leq \int_a^b \max_{(x,y)} \left| \frac{\partial G}{\partial x}(x,y)\right|\;dx \leq (b-a) \max_{x\in [a,b]} \int_0^1 \left|\frac{\partial F}{\partial x}(x,y)\right|\;dy
\end{align*}
and hence the proof is complete.
\end{proof}

\begin{proof}[Proof of Lemma \ref{key lemma}] For each $a\in [0,1]$, an $\delta$-neighborhood $\mathcal{N}_{a,\delta}$ of $a$ is defined as $(a-\delta,a+\delta)$ if $a\in (0,1)$ and $[0,\delta)\cup (1-\delta,1]$ if $a \in \{0,1\}$. It is clear that $\mathcal{N}_{a,\delta}$ is open in $[0,1]$. We claim that there exists $\delta = \delta(a)>0$ such that 
\begin{equation}\label{key}
\sup_{x\in \mathcal{N}^*_{a,\delta}}\frac{\left|\mathcal{V}\,'(x)\right|}{\sqrt{\mathcal{V}(x)}} \leq C_a < \infty
\end{equation}
for some constant $C_a$, where $ \mathcal{N}^*_{a,\delta} = \{x\in \mathcal{N}_{a,\delta}: \mathcal{V}(x)\neq 0\}$. Assume that \eqref{key} is false, then there exists a sequence $x_k\rightarrow a^+$ such that $\mathcal{V}(x_k) \neq 0$ for all $k\in \mathbb{N}$ and
\begin{equation}\label{key.1}
\lim_{k\rightarrow \infty} \frac{\left|\mathcal{V}\,'(x_k)\right|}{\sqrt{\mathcal{V}(x_k)}} = +\infty.
\end{equation}
It is clear that $\mathcal{V}\,'(x_k)\neq 0$ for all $k\in \mathbb{N}$. We can assume that $\mathcal{V}\,'(x_k) >0$ for all $k$. Let $g_k = \sqrt{\mathcal{V}(x_k)}$ and $h_k = \mathcal{V}\,'(x_k)$, and 
\begin{equation*}
a_k = \sup \left\lbrace r>0: \mathcal{V}\,'(x)\geq \frac{h_k}{2}\;\text{for all}\; x\in (x_k-r,x_k) \right\rbrace.
\end{equation*}
Clearly $\mathcal{V}\,'(x_k-a_k) = \frac{h_k}{2}$. By mean value theorem we have
\begin{equation*}
g_k^2=  \mathcal{V}(x_k)\geq \mathcal{V}(x_k) - \mathcal{V}(x_k-a_k) \geq \frac{1}{2}h_k a_k.
\end{equation*}
By mean value theorem again, there exists $\xi_k\in (x_k-a_k,x_k)$ such that
\begin{align*}
\mathcal{V}\,''(\xi_k)&=\frac{\mathcal{V}\,'(x_k)-\mathcal{V}\,'(x_k-a_k)}{a_k}= \frac{1}{2}\frac{h_k}{a_k}\geq 
\frac{1}{4} \left(\frac{h_k}{g_k}\right)^2\rightarrow \infty
\end{align*}
as $k\rightarrow \infty$ due to \eqref{key.1}. It is a contradiction since $\mathcal{V}\in \mathrm{C}^2([0,1])$, thus \eqref{key} must be correct. By compactness of $[0,1]$, we can pick a finite subcover of $[0,1]$ from the open cover $\left\lbrace \mathcal{N}_{a,\delta}:a\in [0,1]\right\rbrace$. From \eqref{key} there exists a constant $L>0$ such that
\begin{equation}\label{key.2}
\frac{\left|\mathcal{V}\,'(x)\right|}{\sqrt{\mathcal{V}(x)}} \leq L \qquad\text{whenever}\qquad \mathcal{V}(x)\neq 0.
\end{equation}
For $0<\varepsilon<1$ let $f_\varepsilon(x) = \sqrt{\mathcal{V}(x)+\varepsilon} \in \mathrm{C}^2([0,1])$. It $\mathcal{V}(x) = 0$ then $\mathcal{V}\,'(x) = 0$, hence $f'_\varepsilon(x) = 0$ as well, while if $\mathcal{V}(x)\neq 0$ then from \eqref{key.2} we have 
\begin{equation*}
|f'_\varepsilon(x)| = \left|\frac{\mathcal{V}\,'(x)}{2\sqrt{\mathcal{V}(x)+\varepsilon}}\right| \leq \frac{1}{2}\frac{\left|\mathcal{V}\,'(x)\right|}{\sqrt{\mathcal{V}(x)}} \leq \frac{L}{2}
\end{equation*}
Thus $f_\varepsilon(x)$ is Lipschitz on $[0,1]$ with a Lipschitz constant independent of $\varepsilon$. Let $\varepsilon\rightarrow 0$ we deduce that $x\mapsto\sqrt{\mathcal{V}(x)}$ is Lipschitz on $[0,1]$.
\end{proof}

\begin{proof}[Proof of Lemma \ref{key lemma 2}] \quad
\begin{itemize}
\item[(i)] It suffices to show for $G_1$ since the argument is similar for $G_2$. For simplicity, let us denote $G_1,H^{-1}_1$ by $G,H^{-1}$. For $0<\varepsilon<1$ let $f_\varepsilon(x) = G\left(\mathcal{V}(x) +  \varepsilon\right)$ then $f_\varepsilon\in \mathrm{C}^2([0,1])$ and 
\begin{align*}
f'_\varepsilon(x) &= \frac{\mathcal{V}\,'(x)}{\sqrt{\mathcal{V}(x)+\varepsilon}}\left(\frac{H''\big(H^{-1}\left(\mathcal{V}(x)+\varepsilon\right)\big)}{H'\big(H^{-1}\left(\mathcal{V}(x)+\varepsilon\right)\big)}\sqrt{\mathcal{V}(x)+\varepsilon}\right).
\end{align*}
For $x\in [0,1]$ such that $\mathcal{V}(x) = \mathcal{V}\,'(x) = 0$ then obviously $f'_\varepsilon(x) = 0$, while if $x\in [0,1]$ such that $\mathcal{V}(x)\neq 0$ then from \eqref{lemA1} and Lemma \ref{key lemma} we have
\begin{equation*}
\qquad\;\;|f'_\varepsilon(x)|\leq \left|\frac{\mathcal{V}\,'(x)}{\sqrt{\mathcal{V}(x)}}\right|. \left|\frac{H''(\xi)}{H'(\xi)}\sqrt{H(\xi)} \right| \leq L\left(\sup_{\left[0,p^*\right]}\left|\frac{H''(p)}{H'(p)}\sqrt{H(p)}\right|\right)<\infty
\end{equation*}
where $\xi = H^{-1}(\mathcal{V}(x)+\varepsilon)$ and $p^* = H^{-1}(\Vert V\Vert_{L^\infty}+1)$. Therefore $f_\varepsilon$ is Lipschitz on $[0,1]$ with a Lipschitz constant independent of $\varepsilon>0$. Let $\varepsilon\rightarrow 0$ we deduce that $x\mapsto G(\mathcal{V}(x))$ is Lipschitz on $[0,1]$.
\item[(ii)] If $H''(0) > 0$ then there exists $\delta>0$ so that $H''(p)\geq c>0$ for $p\in (-\delta,\delta)$, thus there are some $m,M>0$ such that
\begin{equation}\label{ijk}
m|p|^2 \leq H(p) \leq M|p|^2 \qquad\text{and}\qquad m|p|\leq |H'(p)| \leq M|p|.
\end{equation}
From that \eqref{lemA1_pt1} follows. On the other hand, since $G_i(x) = H'(H_i^{-1}(x))$ for $i=1,2$ and \eqref{ijk} we deduce that for all $x$ small then
\begin{equation}\label{aaa1}
\sqrt{\frac{m}{M}}\sqrt{x}\leq |H'(H_i^{-1}(x))| \leq \sqrt{\frac{M}{m}}\sqrt{x}.
\end{equation}
Since $G_i(x)  = 0 $ if and only if $x=0$, we have \eqref{aaa1} is true for any bounded set of $\mathbb{R}$ after modifying the upper bound and lower bound.
\item[(iii)] Using the convexity of $H$ we have $H(p) \leq p H'(p)$ for all $p$, hence
\begin{equation}\label{3-stars}
\left|\frac{H''(p)}{H'(p)}\sqrt{H(p)}\right| \leq \frac{H''(p)}{\sqrt{|H'(p)|}}\sqrt{|p|}.
\end{equation}
Let $g(p) = H'(p)\in \mathrm{C}^2(\mathbb{R})$ is strictly increasing on $(0,\infty)$ and is strictly decreasing on $(-\infty,0)$ with $g(0) = 0$, we claim that indeed
\begin{equation}\label{4-stars}
\limsup_{p\rightarrow 0}
\frac{g'(p)}{\sqrt{|g(p)|}}  < \infty.
\end{equation}
This can be done by a similar argument to Lemma \ref{key lemma}, hence \eqref{lemA1_pt3} follows.
\item[(iv)] It is clear from direct computation.
\end{itemize}
\end{proof}

\section{Acknowledgments}
The author would like to express his appreciation to his advisor, Hung V. Tran for giving him this interesting problem and for his invaluable guidance. The author would like to thanks the referees for invaluable comments and suggestions, which help much in vastly improving the presentation of the paper. The author is grateful to thanks Jingrui Cheng and Ilyas Khan for many useful suggestions, and Truong-Son Van for introducing the reference \cite{Neukamm2017}. Finally, the author also would like to thank Sigurd Angenent, Jean-Luc Thiffeault, Joseph Jepson, Minh-Binh Tran and Michel Alexis for helpful discussions and supports.

%\nocite{*}

\bibliography{zzzzlibrary}{}
\bibliographystyle{acm}

\end{document}